\documentclass[reqno,11pt]{amsart}
\usepackage{amssymb, amsmath,latexsym,amsfonts,amsbsy, amsthm}
\usepackage{color}
\usepackage{mathrsfs}
\usepackage{esint}
\usepackage{graphics,color}
\usepackage{enumerate}
\usepackage{marginnote}

\newcommand{\beq}{\begin{equation}}
\newcommand{\eeq}{\end{equation}}
\newcommand{\ben}{\begin{eqnarray}}
\newcommand{\een}{\end{eqnarray}}
\newcommand{\beno}{\begin{eqnarray*}}
\newcommand{\eeno}{\end{eqnarray*}}

\renewcommand{\theequation}{\thesection.\arabic{equation}}

\newtheorem{theorem}{Theorem}[section]

\newtheorem{proposition}[theorem]{Proposition}

\newtheorem{Theorem}{Theorem}[section]
\newtheorem{Definition}[Theorem]{Definition}
\newtheorem{Proposition}[Theorem]{Proposition}
\newtheorem{Lemma}[Theorem]{Lemma}
\newtheorem{Corollary}[Theorem]{Corollary}
\newtheorem{Remark}[Theorem]{Remark}

\begin{document}

\title[Boussinesq Equation]
{Finite energy weak solutions of 2d Boussinesq Equations with diffusive temperature}

\author{Tianwen Luo}
\address{School of  Mathematical Sciences, Yau Mathematical Science center, Tsinghua University, Beijing 100084, China}
\email{twluo@mail.tsinghua.edu.cn}

\author{Tao Tao}
\address{School of  Mathematics Sciences, Shandong University, jinan, 250100, China}
\email{taotao@amss.ac.cn}
\author{Liqun Zhang}
\address{Academy of Mathematic and System Science , CAS Beijing 100190, China; and School of Mathematical Sciences, University of Chinese Academy of Sciences, Beijing 100049, China}
\email{lqzhang@math.ac.cn}

%

\date{\today}
\maketitle

\renewcommand{\theequation}{\thesection.\arabic{equation}}
\setcounter{equation}{0}

\begin{abstract}
We show the existence of finite kinetic energy solution with prescribed kinetic energy to the 2d Boussinesq equations with diffusive temperature on torus.
\end{abstract}

\noindent {\sl Keywords: Finite energy weak solution, Boussinesq equation, prescribed kinetic energy}

\vskip 0.2cm

\noindent {\sl AMS Subject Classification (2000):} 35Q30, 76D03  \

\setcounter{equation}{0}

\section{Introduction}
The Boussinesq equation was introduced for understanding the effect of potentially large conversions between internal energy and mechanical energy in fluids, and simulates many geophysical flows, such as atmospheric fronts and ocean circulations (see, for example, \cite{Ma},\cite{Pe}). Moreover, it was used in recent theoretical discussion of the energetics of horizontal convection and the energetics of turbulent mixing in stratified fluids.

In this paper, we consider the following 2-dimensional Boussinesq system
\begin{equation}\label{e:boussinesq equation with diffusive temperature}
\begin{cases}
\partial_tv+v\cdot\nabla v+\nabla p+(-\Delta)^\alpha v=\theta e_{2}, \quad {\rm in} \quad [0,1]\times {\rm T}^2\\
\hbox{div}v=0,\quad {\rm in} \quad [0,1]\times {\rm T}^2\\
\partial_t\theta+v\cdot\nabla\theta-\triangle \theta=0, \quad {\rm in} \quad [0,1]\times {\rm T}^2
\end{cases}
\end{equation}
 where $\alpha< 1$ is a positive number, $e_2=(0,1)^T$ and ${\rm T}^2$ is 2d torus. Here, $v$ is the velocity vector, $p$ is the pressure, $\theta$ denotes the temperature which is a scalar function.

The global well-posedness have been established by many authors for the Cauchy problem of (\ref{e:boussinesq equation with diffusive temperature}) in 2d for regularity data(see, for example, \cite{Chae}, \cite{Hou}). For the 3-dimensional case, the global existence of  smooth solution of (\ref{e:boussinesq equation with diffusive temperature}) remains open. To understand the turbulence phenomena in hydrodynamics, one needs to go beyond classical solutions, and in this paper we are interested in constructing weak solutions of (\ref{e:boussinesq equation with diffusive temperature}) with bounded kinetic energy.
 The triple $(v, p, \theta)$ on $[0,1]\times {\rm T}^2$ is called a weak solution of (\ref{e:boussinesq equation with diffusive temperature}) if they belong to $L^\infty(0,1; L^2( {\rm T}^2))$ and solve (\ref{e:boussinesq equation with diffusive temperature})  in the following sense:
\begin{align}
\int_0^1\int_{{\rm T}^2}(v\cdot\partial_t\varphi+v\otimes v:\nabla\varphi-(-\Delta)^\alpha \varphi \cdot u+\theta e_2\cdot\varphi )dxdt=0\nonumber
\end{align}
for all $\varphi\in C_c^\infty((0,1)\times {\rm T}^2;R^2)$ with ${\rm div}\varphi(t,x)=0.$
\begin{align}
\int_0^1\int_{{\rm T}^2}(\partial_t\phi\theta+v\cdot\nabla\phi\theta+\triangle \phi \theta)dxdt=0\nonumber
\end{align}
for all $\phi\in C_c^\infty((0,1)\times {\rm T}^2;R)$ and
\begin{align}
\int_{{\rm T}^2}v(t,x)\cdot\nabla\psi(x) dx=0\nonumber
\end{align}
for all $\psi\in C^\infty( {\rm T}^2;R)$ and any $t\in [0,1]$.

The study of constructing non-unique or dissipative weak solution to fluid system is very fashionable in recent years, and the construction is based on convex integration method which pioneered by De Lellis-Sz\'{e}kelyhidi Jr in \cite{CDL, CDL2}, where the author tackle the Onsager conjecture for the incompressible Euler equation. So far, there are many important work about weak solution of the incompressible Euler equation, see \cite{CPFR, CET, CDL0, DUR, ISOH1, IS1, ONS, VS, ASH1, ASH2, SH, Sz}. The Onsager conjecture was proved by P.Isett in \cite{IS2}, based on a series of process on this problem in \cite{TBU, BCDLI, BCDL1, BCDL2, CHO, DA, DAL, CDL3, ISOH2}, see also
\cite{BCDLV} for the construction  of admissible weak solution. Moreover, the idea and method can be used to construct dissipative weak solution for other model, see \cite{BSV, ISV2, LX, RSH, TZ,TZ1,TZ2}.

 Recently, Buckmaster and Vicol establish the non-uniqueness of weak solution to the 3D incompressible Navier-Stokes in \cite{BV} by introducing some new ideas. The main idea is to use ``intermittent" building blocks in the convex integration scheme to control the dissipative term $\triangle v$, called ``intermittent Beltrami flow", which are space inhomogeneous version of the classical Beltrami flow, also see \cite{BCV}. Compared with the homogenous case, the ``intermittent Beltrami flow" has different scaling in different $L^p$ norms. In particular, one can ensure small $L^p$ norm for small $p>1$ which is key to control the dissipative term. By choosing the parameter suitably, T.Luo and E.S.Titi in \cite{LT} construct weak solution with compact support in time for hyperviscous Navier-Stokes equation. For high dimension$(d\geq 4)$ stationary Navier-Stokes equation, X. Luo in \cite{L} show the non-uniqueness by constructing the concentrated Mikado flows introduced in \cite{MS1, MS2}. Moreover, S.Modena and Sz\'{e}kelyhidi established the non-uniqueness for the linear transport equation (and transport-diffusion equation) with divergence-free vector in some Sobolev space, see \cite{MS1, MS2}, where they use Mikado density and Mikado fields which is highly concentrated such that the $L^p$ norm of Mikado field is small for $p> 1$ small.

Motivated by the above earlier works, we consider the 2d Boussinesq equations (\ref{e:boussinesq equation with diffusive temperature}) and want to know if the similar phenomena can also happen when adding the temperature effects.  Following the general scheme in the construction of non-uniqueness to Navier-Stokes equation in \cite{BV}, we obtain the following existence result.

\begin{Theorem}\label{t:main theorem}
For any smooth function $e(t): [0,1]\rightarrow [1,+\infty)$ and $\frac12\leq \alpha < 1$, there exist $v\in C([0,1], L^2({\rm T}^2)), \theta\in \bigcap_{2\leq p< \infty}C([0,1], L^p({\rm T}^2))\cap L^2([0,1], H^1({\rm T}^2))$, which is weak solution of Boussinesq equation (\ref{e:boussinesq equation with diffusive temperature}) with
\beno
\int_{{\rm T}^2}|v(x,t)|^2 dx=e(t), \quad \forall t \in [0, 1],
\eeno
and for any $t\in [0,1]$
\beno
\frac12\|\theta(t,\cdot)\|_{L^2}^2+\int_0^t\|\nabla \theta(s,\cdot)\|^2_{L^2}ds=\frac12\|\theta(0,\cdot)\|_{L^2}^2.
\eeno
\end{Theorem}

\begin{Remark}
For $0\leq \alpha < \frac12$, we also can construct weak solution with prescribed energy curve in the class $C([0,1], L^2(T^2))$ by the same method. However, in a separate paper \cite{LTZ}, we will construct H\"{o}lder continuous solution for this case.
\end{Remark}

\begin{Remark}
The $\theta$ we construct in this paper is rather regular and satisfies energy equality.
When $\theta=0$, the equation (\ref{e:boussinesq equation with diffusive temperature}) is 2d Navier-Stokes equation with fractional diffusion, and our construction also work for this case.
\end{Remark}

The rest of the paper is organized as follows. In Section 2, we state the main proposition and give a proof of Theorem \ref{t:main theorem}. In Section 3, we collect some technical tool which will be used frequently. In Section 4, we introduce the intermittent plane wave which is the building block in our perturbation. In Section 5 and 6, we construct velocity perturbations and temperature perturbation, respectively. After the construction, we establish the related estimates. In Section 7, we construct the Reynold-Stress and establish the related estimates.  Finally, in Section 8, we give a proof of Proposition \ref{p:iterative proposition}.

\section{Main proposition and proof of main theorem}

In this section, we state our main iterative proposition and give a proof of theorem \ref{t:main theorem} by the help of main proposition.

\begin{Definition}
Assume that $(v_0, p_0, \theta_0, \mathring{R}_0)\in C^\infty([0,1]\times {\rm T}^2, R^2\times R\times R\times S_0^{2\times 2})$. We say that they solve the Boussinesq-Reynold equation if
\begin{equation}\label{e:Boussinesq-Reynold equation}
\left\{
\begin{array}{ll}
\partial_tv_0+{\rm div}(v_0\otimes v_0)+\nabla p_0+(-\triangle)^\alpha v_0=\theta e_{2}+ {\rm div}\mathring{R}_0,\quad\quad \mbox{in}\quad [0,1]\times {\rm T}^2\\[5pt]
{\rm div}v_0=0, \quad\quad \mbox{in}\quad [0,1]\times {\rm T}^2 \\[5pt]
\partial_t\theta_0+{\rm div}(v_0\theta_0)-\triangle \theta_0=0, \quad\quad \mbox{in}\quad [0,1]\times {\rm T}^2.
\end{array}
\right.
\end{equation}
\end{Definition}
Here and throughout the paper, $S_0^{2\times 2}$ is the set of trace-free symmetric $2\times 2$ matrices.

We now state our main proposition, and Theorem \ref{t:main theorem} is a corollary.
\begin{Proposition}\label{p:iterative proposition}
Let $e(t), \alpha$ be as in Theorem \ref{t:main theorem} and $\varepsilon_0$ be a universal constant from the Geometric Lemma \ref{l:geometric lemma}. Then there exist universal constant $M$ and $\tilde{M}$ such that the following hold.

Let $\delta\leq 1$ be any positive number, $\theta^0\in C^\infty({\rm T}^2)$ be any function satisfies $\fint_{{\rm T}^2}\theta^0(x)dx=0$,  and  $(v_0, p_0, \theta_0, \mathring{R}_0)$ is a solution of Boussinesq-Reynold equation (\ref{e:Boussinesq-Reynold equation}) with
\ben\label{e:assumption on error 1}
\|\mathring{R}_0\|_{L^\infty_t L^1_x}\leq \frac{\varepsilon_0\delta}{10000},\quad \fint_{{\rm T}^2}\theta_0(t,x)dx=0, \quad \forall t\in [0, 1]
\een
and
\ben\label{e:assumtion on energy 1}
\frac{3\delta}{4} e(t)\leq e(t)-\int_{{\rm T}^2}|v_0(t,x)|^2dx\leq \frac{5\delta} {4} e(t), \quad \forall t\in [0, 1].
\een

Then there exist another smooth functions $(v_1, p_1, \theta_1, \mathring{R}_1)$ which is also a solution of Boussinesq-Reynold equation (\ref{e:Boussinesq-Reynold equation}), and for every $ t\in [0, 1]$,
\ben\label{e: estimate for new constructing function}
\|\mathring{R}_1(t,\cdot)\|_{ L^1_x}&\leq& \frac{\varepsilon_0\delta}{20000}, \quad \fint_{{\rm T}^2}\theta_1(t,x)dx=0, \nonumber\\
 \|v_1(t,\cdot)-v_0(t,\cdot)\|_{L^2}&\leq& M\sqrt{\delta},\nonumber\\[5pt]
 \theta_1(0,x)=\theta^0(x), &&\quad \|\theta_1(t,\cdot)\|_{L^\infty}\leq \tilde{M},\nonumber\\
\|\theta_1(t,\cdot)-\theta_0(t,\cdot)\|^2_{L^2}&+&\int_0^t \|\nabla\theta_1(s,\cdot)-\nabla\theta_0(s,\cdot)\|^2_{L^2}\leq \tilde{M}\delta,\nonumber\\
\|\theta_1(t,\cdot)\|^2_{L^2}&+&2\int_0^t \|\nabla\theta_1(s,\cdot)\|^2_{L^2}=\|\theta_1(0,\cdot)\|^2_{L^2},
\een
and
\ben\label{e:energy estimate new}
\frac{3 \delta}{8}e(t)\leq e(t)-\int_{{\rm T}^2}|v_1(t,x)|^2dx\leq \frac{5\delta}{8}e(t).
\een
\end{Proposition}

We will prove the above proposition in the next several sections. Here we first give a proof of Theorem \ref{t:main theorem}.
\begin{proof}
We first fix $\delta=1$ and set
\begin{align}
v_{0}:=0, \quad \theta_0:=0, \quad p_{0}:=0,\quad \mathring{R}_{0}:=0
.\nonumber
\end{align}
Obviously, they solve Boussinesq-Reynolds system (\ref{e:Boussinesq-Reynold equation}) and
\begin{align}
&\frac{3\delta}{4}e(t)\leq e(t)-\int_{{\rm T}^2}|v_{0}|^{2}(x,t)dx\leq\frac{5\delta}{4}e(t),\qquad \forall t \in [0,1]\nonumber\\
&    \sup_{t\in [0,1]}\|\mathring{R}_{0}(t,\cdot)\|_{L^1_x}=0\leq \frac{\varepsilon_0\delta}{10000}.\nonumber
\end{align}
Then choosing $\theta^0\in C^\infty({\rm T}^2)$ satisfies $\fint_{{\rm T}^2}\theta^0(x)dx=0$, and using Proposition \ref{p:iterative proposition} iteratively, we can construct a sequence $(v_{n},~p_{n},~\theta_{n},~\mathring{R_{n}})$,
which solve (\ref{e:Boussinesq-Reynold equation}) and satisfy, for every $t \in [0,1]$
\begin{align}
    \frac{3}{4}\frac{e(t)}{2^{n}}\leq e(t)-\int_{{\rm T}^2}|v_{n}|^{2}(x,t)dx\leq& \frac{5}{4}\frac{e(t)}{2^{n}},\label{e:energy_final}\\
     \|\mathring{R_{n}}(t,\cdot)\|_{L^1_x}\leq&\frac{\varepsilon_0}{2^n\times 10000} ,\label{e:stress_final}\\
    \|v_{n+1}(t,\cdot)-v_{n}(t,\cdot)\|_{L^2} \leq& M\sqrt{\frac{1}{2^{n}}},\label{e:velocity_final}\\
     \|\theta_n(t,\cdot)\|_{L^\infty}\leq& \tilde{M},\quad \theta_n(0,x)=\theta^0(x), \\
\|\theta_{n+1}(t,\cdot)-\theta_n(t,\cdot)\|^2_{L^2}+&\int_0^t \|\nabla(\theta_{n+1}(s,\cdot)-\theta_n(s,\cdot))\|^2_{L^2}\leq \frac{\tilde{M}}{2^n},\\
\|\theta_{n+1}(t,\cdot)\|^2_{L^2}+&2\int_0^t \|\nabla\theta_{n+1}(s,\cdot)\|^2_{L^2}=\|\theta_{n+1}(0,\cdot)\|^2_{L^2}.\label{e:temperature_final}
\end{align}
From (\ref{e:stress_final})-(\ref{e:temperature_final}), we know that $(v_{n},~\theta_{n},~\mathring{R}_{n})$ are, respectively,  Cauchy sequence in $C([0,1], L^2({\rm T}^2))$, $\bigcap_{2\leq p< \infty}C([0,1], L^p({\rm T}^2))\cap L^2([0,1], H^1({\rm T}^2))$ and $C([0,1], L^1({\rm T}^2))$, therefore there exist
\begin{align}
    v\in C([0,1], L^2({\rm T}^2)), \quad \theta \in \bigcap_{2\leq p< \infty}C([0,1], L^p({\rm T}^2))\cap L^2([0,1], H^1({\rm T}^2))\nonumber
\end{align}
such that
\begin{align}
   & v_{n}\rightarrow v \quad {\rm in} \quad C([0,1], L^2({\rm T}^2)),\nonumber\\[5pt]
    & \theta_{n}\rightarrow \theta \quad {\rm in} \quad \bigcap_{2\leq p< \infty}C([0,1], L^p({\rm T}^2))\cap L^2([0,1], H^1({\rm T}^2)), \nonumber\\
    & \mathring{R_{n}}\rightarrow 0 \quad {\rm in} \quad C([0,1], L^1({\rm T}^2))\nonumber
\end{align}
as $n \rightarrow \infty$.\\
Passing into the limit in (\ref{e:Boussinesq-Reynold equation}), we conclude that $(v,~\theta)$ solve (\ref{e:boussinesq equation with diffusive temperature}) in the sense of distribution.
Moreover,~by \eqref{e:energy_final},
$$ e(t)=\int_{{\rm T}^2}|v|^{2}(x,t)dx, \qquad \forall t\in[0,1].$$
Moreover, by (\ref{e:temperature_final}), we deduce that the temperature $\theta$ satisfies the energy equality: for every $t\in [0, 1]$
\beno
\|\theta(t,\cdot)\|^2_{L^2}+2\int_0^t \|\nabla\theta(s,\cdot)\|^2_{L^2}=\|\theta(0,\cdot)\|^2_{L^2}.
\eeno
This complete the proof of Theorem \ref{t:main theorem}.
\end{proof}

The rest of this paper will be dedicated to prove Proposition \ref{p:iterative proposition}.
First, we add perturbations to $v_0$  and get new functions $v_{1}$ as following:
 \begin{align}
 v_{1}=&v_0+w^p_{1}+w^c_{1}+w^t_{1}:=v_0+w_{1},\nonumber
 \end{align}
where $w^p_{1}, w^c_{1}, w^t_{1}$ are smooth functions given by explicit formulas.  We introduce some parameters $\lambda_1, \mu_1, r_1, \sigma_1$  satisfying the relation (\ref{r:relationship of papameter}).
After the construction of new velocity $v_{1}$, we construction new temperature $\theta_{1}$ by solving the following transport-diffusion equation: there exists $\theta_{1}\in C^\infty([0, 1]\times {\rm T}^2, R)$ which solves
\begin{equation}\label{e:transport-diffusion equation}
\begin{cases}
\partial_t \theta_{1}+ v_{1}\cdot\nabla  \theta_{1} - \triangle  \theta_{1}=0,\\[3pt]
\theta_{1}|_{t=0}=\theta^0,
\end{cases}
\end{equation}
where $\theta^0$ is the function appeared in Proposition \ref{p:iterative proposition}.
After construction of $v_{1}, \theta_{1}$, we mainly focus on finding functions $\mathring{R}_{1}, p_{1}$ with the desired estimates and solving system (\ref{e:Boussinesq-Reynold equation}).

%

\section{Technical tool}

In this section, we collect some technical tools which will be frequently used in the following.

\subsection{Properties of fast oscillatory}

In this subsection, we discuss some properties of fast oscillatory and the proof can be found in \cite{MS1}, \cite{MS2}, which was inspired by \cite{BV}. More precisely, we give an improved H\"{o}lder inequality which concern the $L^p$ norm of the product of a slow oscillating function with a fast oscillating function,   and a mean value estimate which concern the mean value of the product of a slow oscillating function with a fast oscillating function .

For a given function $f: {\rm T}^2\rightarrow R$, $\lambda\in {\rm N}$,  we set
\beno
f_\lambda(x):=f(\lambda x).
\eeno
\begin{Lemma}\label{e:improved $L^p$ inequality}
Let $f,g: {\rm T}^2\rightarrow R$ be smooth functions, $\lambda \in {\rm N}$. Then for every $p\in [1, +\infty]$, we have
\beno
\|f g_\lambda\|_{L^p}\leq \|f\|_{L^p}\|g\|_{L^p}+\frac{C_p}{\lambda^{\frac{1}{p}}}\|f\|_{C^1}\|g\|_{L^p}.
\eeno
\end{Lemma}


\begin{Lemma}\label{e:mean value estimate}
Let $f, g: {\rm T}^2\rightarrow R$ be smooth function with $\fint_{{\rm T}^2}g(x)dx=0$, and $\lambda\in {\rm N}.$ Then there hold
\beno
\Big|\fint_{{\rm T}^2}fg_\lambda dx\Big|\leq \frac{\sqrt{2} \|f\|_{C^1}\|g\|_{L^1}}{\lambda}.
\eeno
\end{Lemma}

\subsection{Commutator for fast oscillation}

\begin{Lemma}\label{e:oscillatory lemma}
	Fix $\kappa\geq 1$. Let $a\in C^2(\mathbb{T}^3)$.
	For $1 < p < \infty$, and any $f \in L^p(\mathbb{T}^3)$, we have
	\begin{align*}
	\| |\nabla|^{-1}\mathbb{P}_{\neq 0}(a \mathbb{P}_{\geq k} f)\|_{L^p} &\lesssim k^{-1}  (\|a\|_{L^{\infty}}  + \|\nabla^2 a\|_{L^{\infty}}) \|f\|_{L^p}. 
	\end{align*}
\end{Lemma}

The proof of this Lemma can be found in \cite{BV}.

\section{Intermittent plane waves}

In this section, we describe in detail the construction of the \textit{intermittent plane waves} which will form the building blocks of the convex integration scheme.

We first recall the following stationary solution for the 2d Euler equation. Our building block in this paper is the inhomogeneous version of it.
\subsection{Stationary flows in 2D and Geometric Lemma}
\begin{Proposition}\label{Prop:Beltrami-2D}
Let $\Lambda$ be a given finite symmetric subset of $S^1\cap Q^2$ with $\lambda\Lambda \in Z^2.$ Then for any choice of coefficients $a_{\bar{\xi}}\in {\rm C}$ with $\overline{a_{\bar{\xi}}}=a_{-\bar{\xi}}$, the vector field
\beno
W(x)=\sum_{\bar{\xi}\in \Lambda}a_{\bar{\xi}}i\bar{\xi}^{\bot}e^{i \bar{\xi}\cdot x}, \quad \Psi(x)=\sum_{\bar{\xi}\in \Lambda}a_{\bar{\xi}}e^{i \bar{\xi}\cdot x}
\eeno
is real-valued and satisfies
\beno
{\rm div}(W\otimes W)=\nabla\Big(\frac{|W|^2}{2}+\frac{\Psi^2}{2}\Big), \quad W(x)=\nabla^{\bot}\Psi(x).
\eeno
Here and throughout the paper, we denote $\bar{\xi}^{\perp}=(-\bar{\xi}_2, \bar{\xi}_1)$ if $\bar{\xi}=( \bar{\xi}_1, \bar{\xi}_2)$, and denote $\nabla^\bot=(-\partial_2, \partial_1)$.
Furthermore,
\beno
\big<W\otimes W\big>:=\fint_{{\rm T}^2}W\otimes W(x)dx=\sum_{\bar{\xi}\in \Lambda}|a_{\bar{\xi}}|^2({\rm Id}-\bar{\xi}\otimes \bar{\xi}).
\eeno
\end{Proposition}
The proof of this Proposition can be found in \cite{CHO}, and we omit it here.

Let
\begin{align*}
\Lambda_0^+ = \Big\{ e_1, \frac{3}{5}e_1 + \frac{4}{5}e_2, \frac{3}{5}e_1 - \frac{4}{5}e_2  \Big\}, \quad \Lambda_0^- = - \Lambda_0^+, \quad \Lambda_0 = \Lambda_0^+ \cup \Lambda_0^-,
\end{align*}
and $\Lambda_1$ be given by the rotation of $\Lambda_0$ counter clock-wise by $\pi/2$:
\begin{align*}
\Lambda_1^+ = \Big\{ e_2, \frac{3}{5}e_2 + \frac{4}{5}e_1, \frac{3}{5}e_2 - \frac{4}{5}e_1  \Big\}, \quad \Lambda_1^- = - \Lambda_1^+, \quad \Lambda_1 = \Lambda_1^+ \cup \Lambda_1^-.
\end{align*}
Clearly $\Lambda_0, \Lambda_1 \subseteq  Q^2\cap S^1$ and we have the representation
\begin{align*}
\frac{25}{32}\Big(\Big(\frac{3}{5}e_1 + \frac{4}{5}e_2\Big)\otimes \Big(\frac{3}{5}e_1 + \frac{4}{5}e_2\Big)+\Big( \frac{3}{5}e_1 - \frac{4}{5}e_2\Big)\otimes  \Big(\frac{3}{5}e_1 - \frac{4}{5}e_2\Big)\Big) + \frac{7}{16}e_1\otimes e_1 = \mathrm{Id}.
\end{align*}
In fact, such representation holds for $2 \times 2$ symmetric matrices near $\mathrm{Id}$.
\begin{Lemma}[Geometric Lemma]\label{l:geometric lemma}
There exists $\varepsilon_{0}>0$, and smooth positive functions $\gamma_{\xi}$:
 \begin{align*}
    \gamma_{\xi}\in C^{\infty}(B_{\varepsilon_{0}}({\rm Id})) 
\end{align*}
such that for every $2 \times 2$ symmetric matrix $R \in B_{\varepsilon_{0}}({\rm Id})$, we have
  $$R = \sum_{\xi \in \Lambda_0^+} \gamma^2_{\xi}(R)\xi \otimes \xi.$$
\end{Lemma}

\begin{Remark}\label{r:remark about universal constant}
By rotational symmetry, Geometric Lemma \ref{l:geometric lemma} also holds for $\xi \in \Lambda_1^+$.
It is  convenient to introduce a small geometric constant $c_0\in (0, 1)$ such that
\beno
|\xi+\xi'|\geq 2c_0
\eeno
for all $\xi, \xi' \in \Lambda_0 \cap \Lambda_1, \xi \neq -\xi'$. Moreover, for $\xi\in \Lambda_k^{-}$ with $k=0,1$, we set $\gamma_\xi:=\gamma_{-\xi}$.
\end{Remark}

\begin{Remark}
When ${\rm Id}-\mathring{R} \in B_{\varepsilon_0}({\rm Id})$, by Geometric Lemma, we have
\beno
{\rm Id}-\mathring{R}=\sum_{\xi \in \Lambda_0^+} \gamma_{\xi}^2\big({\rm Id}-\mathring{R}\big)\xi \otimes \xi.
\eeno
Thus, taking trace in both side, we obtain
\ben\label{e:identity in Geomertic constant}
\sum_{\xi \in \Lambda_0^+} \gamma_{\xi}^2\big({\rm Id}-\mathring{R}\big)=2.
\een
\end{Remark}

\subsection{Intermittent plane flow} The Dirichlet kernel $\tilde{D}_r$ is defined as
\beno
\tilde{D}_r(x):=\sum_{\xi=-r}^r e^{ix\cdot \xi}=\frac{{\rm sin}((r+\frac12)x)}{{\rm sin}(\frac{x}{2})},
\eeno
and it obeys the estimates: for any $p> 1$,
\beno
\|\tilde{D}_r\|_{L^p}\sim r^{1-\frac{1}{p}},
\eeno
where the implicit constant only depend only on $p$.
Define a 2d square
\beno
\Omega_r:=\{(j,k): j,k\in \{-r,..., r\}\}
\eeno
and normalizing to unit size in $L^2$, we obtain a kernel
\beno
D_r(x):=\frac{1}{2r+1}\sum_{\xi\in \Omega_r} e^{ix\cdot \xi}=\frac{1}{2r+1}\sum_{(j,k)\in \Omega_r} e^{i(jx_1+kx_2)}
\eeno
which has the property: for $1< p\leq\infty $
\beno
\|D_r\|_{L^p}\lesssim r^{1-\frac{2}{p}},\quad \|D_r\|_{L^2}\thickapprox 1,
\eeno
where the implicit constant only depend on $p$.
This computation is very easy due to the fact:
\beno
\sum_{\xi\in \Omega_r} e^{ix\cdot \xi}=\Big(\sum_{i=-r}^r e^{ijx_1}\Big) \Big( \sum_{i=-r}^r e^{ikx_2}\Big).
\eeno

We first fixed a large parameter $\lambda\in {\rm N}$, then introduce a parameter $\sigma$ such that $\lambda \sigma\in {\rm N}$ which parameterizes the spacing between frequencies. We assume that
\beno
\sigma r\leq \frac{c_0}{50N},
\eeno
where $c_0$ is the constant in Remark \ref{r:remark about universal constant}, and $N$ is a fixed integer(for example, we can take $N=5$).
Furthermore, we introduce a  parameter $\mu\in (0, \lambda)$, which describes temporal oscillation in the building blocks.

As in \cite{BV}, for $\xi\in \Lambda_j^+$, we define a directed and rescaled  $(\frac{2\pi}{\lambda\sigma})^2$-periodic Dirichlet kernel by
\beno
\eta_{(\xi)}(x,t):=\eta_{\xi, \lambda, \sigma, r, \mu}(x,t)&=&D_r\big(\lambda\sigma N (\xi\cdot x+\mu t),\lambda\sigma N \xi^{\bot}\cdot x \big)\\
&=& \Big(\sum_{j=-r}^r e^{ij\lambda\sigma N (\xi\cdot x+\mu t)}\Big)\Big(\sum_{k=-r}^re^{ik\lambda\sigma N \xi^{\bot}\cdot x}\Big)
\eeno
and set $\eta_{(\xi)}(x,t) =  \eta_{(-\xi)}(x,t)$ for $\xi\in \Lambda_j^-$, where $N$ is a integer(we can set $N=5$ duo to our construction of $\Lambda_0, \Lambda_1$).
Observe that $\eta_{(\xi)}(x,t)$ satisfies the following important identity:
\ben\label{e:transport structure}
\frac{1}{\mu}\partial_t \eta_{(\xi)}(x,t)=\xi\cdot \nabla \eta_{(\xi)}(x,t)
\een
for every $\xi\in \Lambda_j^+$, $j=0,1$.

A change of variable gives that
\ben\label{e:property about Dirichlet kernel}
\fint_{{\rm T}^2} \eta^2_{(\xi)}(x,t)dx=1, \quad \|\eta_{(\xi)}(t)\|_{L^p}\leq C_p r^{1-\frac{2}{p}}
\een
for all $1< p \leq\infty$, for any $t$.

Let $W_{(\xi)}(x)$ be the stationary wave at frequency $\lambda$, namely
\beno
W_{(\xi)}(x):=W_{\xi,\lambda}=i\xi e^{i\lambda \xi^{\bot}\cdot x}.
\eeno
We define the \textit{intermittent plane waves} $\textbf{W}_{(\xi)}$ as
\beno
\textbf{W}_{(\xi)}(x,t)=\textbf{W}_{\xi,\lambda,\sigma,r,\mu}(x,t)=\eta_{\xi,\lambda,\sigma,r,\mu}(x,t)W_{\xi,\lambda}(x)
=\eta_{(\xi)}(x,t)W_{(\xi)}(x).
\eeno

\begin{Remark}
The explicit representation of $\textbf{W}_{(\xi)}(x,t)$ is as following:
\beno
\textbf{W}_{(\xi)}(x,t)&=&\Big(\sum_{j=-r}^r e^{ij\lambda\sigma N (\xi\cdot x+\mu t)}\Big)\Big(\sum_{k=-r}^re^{ik\lambda\sigma N \xi^{\bot}\cdot x}\Big)i\xi e^{i\lambda \xi^{\bot}\cdot x}.
\eeno
\end{Remark}

Some facts about the frequency support of $\eta_{(\xi)}$ and $\textbf{W}_{(\xi)}$:
\begin{Lemma}\label{e:frequency support}
We have the following frequency support property:
\beno
{\rm P}_{\leq 2\lambda \sigma r N}\eta_{(\xi)}=\eta_{(\xi)}, \quad {\rm P}_{\leq 2\lambda} {\rm P}_{\geq \frac{\lambda}{2}}\textbf{W}_{(\xi)}=\textbf{W}_{(\xi)}.
\eeno
For $\xi+\xi^{'}\neq 0$, by the definition of $c_0$, we have
\beno
{\rm P}_{\leq 4\lambda}{\rm P}_{\geq c_0 \lambda}\big(\textbf{W}_{(\xi)}\otimes \textbf{W}_{(\xi^{'})}\big)=\textbf{W}_{(\xi)}\otimes \textbf{W}_{(\xi^{'})}.
\eeno
\end{Lemma}
These facts can be obtained directly from the definition. In fact, the frequency support of $\eta_{(\xi)}$ is obvious. Then, using the fact that
$2\lambda \sigma r N\leq \frac{c_0\lambda}{5}$, it's easy to obtain the frequency support of $\textbf{W}_{(\xi)}$. Finally, a direct computation gives that
\beno
\textbf{W}_{(\xi)}\otimes \textbf{W}_{(\xi^{'})}=-\eta_{(\xi)} \eta_{(\xi^{'})}\xi\otimes \xi^{'}e^{i\lambda (\xi+\xi^{'})^{\bot}\cdot x}.
\eeno
Due to
\beno
{\rm P}_{\leq 4\lambda \sigma r N}(\eta_{(\xi)} \eta_{(\xi^{'})})=\eta_{(\xi)} \eta_{(\xi^{'})}, \quad |\xi+\xi^{'}|\geq 2c_0
\eeno
and the fact
\beno
4\lambda \sigma r N\leq \frac{2\lambda c_0}{5},
\eeno
we obtain the frequency support of $\textbf{W}_{(\xi)}\otimes \textbf{W}_{(\xi^{'})}$ for $\xi+\xi^{'}\neq 0$.

From these frequency support properties (from which we can use the Berstein inequality) and the estimates for Dirichlet kernel, we have the following estimates.
\begin{Proposition}\label{e:estimate on building blocks}
Let $\textbf{W}_{(\xi)}$ be defined as above. Then
\begin{align}
&\|\nabla^N \partial_t^K\textbf{W}_{(\xi)} \|_{L^p}\leq C(N,K,p)\lambda^N (\lambda \sigma \mu r)^K r^{1-\frac{2}{p}},\nonumber\\
&\|\nabla^N \partial_t^K \eta_{(\xi)} \|_{L^p}\leq C(N,K,p) (\lambda \sigma r)^N (\lambda \sigma \mu r)^K r^{1-\frac{2}{p}},\nonumber
\end{align}
for any $1< p \leq\infty$, and $N ,K \geq 0$ are integer.
\end{Proposition}
These estimates are direct, and we omit the proof here.



\section{The velocity perturbation: Construction and Estimates}

In this section, we construct the perturbation of velocity and give some estimates for it.
\subsection{Construction of the velocity perturbation} In this subsection, we give the detailed construction of velocity perturbation.

\subsubsection{Definition of amplitude}
Choose two smooth cutoff functions $\tilde{\chi}_0, \tilde{\chi}_1$ such that
\beno
{\rm supp}\tilde{\chi}_0\subseteq [0, 4], \quad {\rm supp}\tilde{\chi}_1\subseteq \Big[\frac14, 4\Big]
\eeno
and
\beno
\tilde{\chi}_0^2(y)+\sum_{j\geq 1}\tilde{\chi}_{j}^2(y)=1
\eeno
for any $y>0$, where $\tilde{\chi}_{j}(y)=\tilde{\chi}(4^{-j}y)$. We then define
\ben\label{d:definition of cutoff function}
\chi_j(t,x):= \tilde{\chi}_{j}\Big(\Big<\frac{50\mathring{R}_0(t,x)}{\varepsilon_0\delta}\Big>\Big)
\een
for all $j\geq 0.$ Here and throughout the paper we use the notation $\big<A\big>=(1+|A|^2)^{\frac12}$ where $|A|$ denotes the standard norm of the matrix $A$. By the definition of the cutoff functions, we have
\ben\label{e:basic property of cutoff function}
\sum_{j\geq 0}\chi_j^2(t,x)=1,\quad \chi_{j_1} \chi_{j_2}(t,x)=0 \quad {\rm if} \quad |j_1-j_2|\geq 2.
\een
Moreover, it's obvious that there exists an index $j_{max}=j_{max}(\mathring{R}_0, \delta)$ such that $\chi_j(t,x)=0$ for all $j\geq j_{max}.$ Another important fact is
\ben\label{e:positive property of zero term}
\int_{{\rm T}^2}\chi^2_0(t,x)dx> 0 , \quad \forall t\in [0,1].
\een
In fact,
\beno
\chi_0(t,x)=0 \Leftrightarrow \Big<\frac{50 \mathring{R}_0(t,x)}{\varepsilon_0\delta}\Big> \geq 4.
\eeno
Thus, for fixed $t\in [0,1]$, we have
\beno
\chi_0(t,x)=0 \Rightarrow 50 |\mathring{R}_0(t,x)|\geq 3\varepsilon_0\delta.
\eeno
Hence, for fixed $t\in [0,1]$, we have
\beno
\int_{{\rm T}^2}|\mathring{R}_0(t,x)|dx\geq \frac{3}{50}\varepsilon_0\delta.
\eeno
Thus, the assumption (\ref{e:assumption on error 1}) tells us that for every $t\in [0,1]$, $\int_{{\rm T}^2}\chi^2_0(t,x)dx> 0.$

For $j \in {\rm N}$, denote
\begin{align*}
\Lambda_{(j)} = \Lambda_{j ~ \text{mod} ~ 2}, \quad \Lambda_{(j)}^{\pm} = \Lambda_{j~ \text{mod} ~2}^{\pm}.
\end{align*}

For $\xi \in \Lambda_{(j)}^+$, define the coefficient function $a_{(\xi,j)}$ by
\ben\label{e:amplitude definition}
a_{(\xi,j)}(t,x):=\sqrt{\rho_j}\chi_j(t,x)\gamma_{\xi}\Big({\rm Id}-\frac{\mathring{R}_0}{\rho_j}\Big)
\een
 where $\rho_j$, $j\geq 1$, are defined by
\ben\label{d:difinition of j-amplitude}
\rho_j:= 4^{j}\delta,
\een
and $\rho_0$ is defined later.

We first claim that $a_{(\xi,j)}(t,x)$ is well-defined for $j\geq 1$. We only need to show that ${\rm Id}-\frac{\mathring{R}_0}{\rho_j}\in B_{\varepsilon_0}({\rm Id})$. In fact, when $\chi_j\neq 0$, there holds
\beno
\Big<\frac{50\mathring{R}_0(t,x)}{\varepsilon_0 \delta}\Big>\leq 4^{j+1},
\eeno
which implies
\beno
\frac{50|\mathring{R}_0(t,x)|}{\varepsilon_0\delta}\leq 4^{j+1}.
\eeno
Thus,
\beno
\frac{|\mathring{R}_0(t,x)|}{\rho_j}=\frac{|\mathring{R}_0(t,x)|}{4^{j}\delta}\leq \varepsilon_0,
\eeno
thus $a_{(\xi,j)}(t,x)$ is well-defined for $j\geq 1$.
We define $\rho_0(t)$ as following:
\ben\label{d:definition of zero amplitude}
\rho_0(t):=\frac{1}{2}\Big(\int_{{\rm T}^2}\chi_0^2(t,x)dx\Big)^{-1}\Big[e(t)\Big(1-\frac{\delta}{2}\Big)-\int_{{\rm T}^2}|v_0(t,x)|^2dx\Big].
\een
Due to (\ref{e:positive property of zero term}), we deduce that $\rho_0(t)$ is well-defined. Next, we show that
 \beno
\Big \|\frac{\mathring{R}_0}{\rho_0}\Big\|_{L^\infty({\rm supp \chi_0)}}\leq \varepsilon_0.
 \eeno
In fact, using the assumption (\ref{e:assumtion on energy 1}), we know that
\ben\label{e:energy difference estimate 1}
\frac{\delta}{4}e(t)\leq e(t)\Big(1-\frac{\delta}{2}\Big)-\int_{{\rm T}^2}|v_0(t,x)|^2dx\leq \frac{3\delta}{4}e(t), \quad \forall t\in [0, 1].
\een
On the support of $\chi_0$, there holds
\beno
50 |\mathring{R}_0(t,x)|\leq 4\varepsilon_0\delta.
\eeno
Thus, on the support of $\chi_0$, for any $t\in [0,1]$, there holds
\beno
\Big|\frac{\mathring{R}_0(t,x)}{\rho_0(t)}\Big|\leq \frac{4\varepsilon_0\delta}{50}\frac{8\int_{{\rm T}^2}\chi_0^2(t,x)dx}{\delta e(t)}\leq \varepsilon_0.
\eeno
Thus, $a_{(\xi, 0)}$ is well-defined.

\subsubsection{Construction of velocity perturbation}
Let us fix $\lambda_1, \sigma_1, r_1, \mu_1$ such that $\lambda_1 \sigma_1  \in {\rm N}$ and the integer $r_1$, the parameter $\sigma_1$ and $\mu_1$ are defined by
\ben\label{r:relationship of papameter}
r_1=[\lambda_1^\alpha], \quad \sigma_1=\lambda_1^{-\frac{1+\alpha}{2}}, \quad \mu_1=[\lambda_1^{\frac{3\alpha+1}{4}}].
\een

The principle part of perturbation $w^p_{1}$  will be defined as
\ben\label{e:principle part}
w_{1}^p = \frac{1}{\sqrt{2}} \sum_{j} \sum_{\xi \in \Lambda_{(j)}^+} a_{(\xi,j)} \eta_{(\xi)}(W_{(\xi)}+W_{(-\xi)}) = \frac{1}{\sqrt{2}} \sum_{j} \sum_{\xi \in \Lambda_{(j)}} a_{(\xi,j)} \eta_{(\xi)}W_{(\xi)}
\een
where $0\leq j\leq j_{max}$. Here and throughout the paper, $\eta_{(\xi)}, W_{(\xi)}$ denote, respectively,
\beno
\eta_{(\xi)}=\eta_{\xi, \lambda_1,\sigma_1,r_1,\mu_1},\quad
 W_{(\xi)} =W_{\xi, \lambda_1},
\eeno
where $\xi \in \Lambda_{(j)}$.

 Then we define an incompressibility corrector
\ben\label{e:corretor part}
w^{c}_{1}:=-\frac{1}{\sqrt{2}}\sum_{j} \sum_{\xi \in \Lambda_{(j)}^+} \nabla^{\bot} \big(a_{(\xi,j)} \eta_{(\xi)}\big)\frac{e^{i \lambda_1 \xi^{\bot}\cdot x}+e^{-i  \lambda_1 \xi^{\bot}\cdot x}}{\lambda_1}.
\een
Here and throughout the paper, we denote $\nabla^{\bot}$ as $\nabla^{\bot}=(-\partial_2, \partial_1).$
Thus, we have
\beno
w^{p}_{1}+w^{c}_{1} = -\frac{1}{\sqrt{2}}\sum_{j} \sum_{\xi \in \Lambda_{(j)}^+} \nabla^{\bot} \Big(a_{(\xi,j)} \eta_{(\xi)}\frac{e^{i \lambda_1 \xi^{\bot}\cdot x}+e^{-i  \lambda_1 \xi^{\bot}\cdot x}}{\lambda}\Big)
\eeno
and
\beno
{\rm div}(w^{p}_{1}+w^{c}_{1})=0.
\eeno
As in paper \cite{BV}, in addition to the incompressibility corrector $w_{1}^{c}$, we introduce a temporal corrector $w_{1}^{t}$, which is defined by
\ben\label{e:defination of time corrector}
w_{1}^{t}:= -\frac{1}{\mu}\sum_{j} \sum_{\xi \in \Lambda_{(j)}^+} P_H P_{\neq 0}\big(a^2_{(\xi,j)}\eta^2_{(\xi)}\xi\big).
\een
Here $P_{\neq0}f=f-\fint_{{\rm T}^2}fdx$ and $P_Hf=f-\nabla\triangle^{-1}{\rm div}f.$
Finally, we define the velocity increment $w_{1}$ by
\beno
w_{1}=w^{p}_{1}+w^{c}_{1}+w^{t}_{1}.
\eeno
It's obvious that
\beno
{\rm div}w_{1}=0,\quad \fint_{{\rm T}^2}w_{1}(t,x)dx=0.
\eeno
After the construction of $w_{1}$, we define the new velocity field $v_{1}$ as
\beno
v_{1}:=v_0+ w_{1}.
\eeno

\subsection{Estimate of the perturbation}In this subsection, we establish some estimates for the velocity perturbation.

Firstly, we collect some estimates concerning the cutoffs function $\chi_j(t,x)$.
\begin{Lemma}\label{e:pointwise estimate on cutoff function}
There exists a $j_{max}=j_{max}(\mathring{R}_0, \delta)$ such that
\beno
\chi_j(t,x)=0, \quad {\rm for \quad all} \quad j> j_{max}.
\eeno
Moreover, for all $0\leq j\leq j_{max}$, there holds
\beno
\rho_j \leq 4^{j_{max}}.
\eeno
\end{Lemma}
\begin{proof}
For $j\geq 1$,
\beno
\chi_j(t,x)\neq 0\Leftrightarrow 4^{j-1}\leq\Big<\frac{50\mathring{R}_0(t,x)}{\varepsilon_0\delta}\Big>\leq 4^{j+1}.
\eeno
Thus, $\chi_j(t,x)\neq 0$ implies
\beno
4^{j-2}\leq  \frac{50|\mathring{R}_0|}{\varepsilon_0\delta}.
\eeno
Thus, there exists $j_{max}=j_{max}(\mathring{R}_0, \delta)$ such that
\beno
\chi_j(t,x)=0, \quad {\rm for \quad all} \quad j> j_{max}.
\eeno
More precisely, we have
\ben\label{e:estimate on number}
4^{j_{max}}\leq \frac{800|\mathring{R}_0|}{\varepsilon_0\delta}.
\een
\end{proof}

\begin{Lemma}\label{e:$L^2$ and derivative estimate on cutoff function}
Let  $0\leq j\leq j_{max}$. There holds
\beno
 \|\chi_j\|_{C^L_{t,x}}\leq C(\mathring{R}_0, \delta,L),
\eeno
where $L$ is integer, and the constant $C$ also depend on $\varepsilon_0$, but $\varepsilon_0$ is a universal constant and we omit it.
\end{Lemma}
\begin{proof}
Direct computation gives that
\beno
\partial_l(\big<A\big>)=\frac{A: \partial_lA}{\big<A\big>},\quad \partial_{ll}(\big<A\big>)=\frac{\partial_lA: \partial_lA+ A:\partial_{ll}A}{\big <A\big>}- \frac{(A:\partial_lA)^2}{\big <A\big>^2}
\eeno
hence
\beno
|\partial_l(\big<A\big>)|\leq |\partial_lA|.
\eeno
Since
\beno
\partial_l(\chi_j(t,x))=\tilde{\chi}'\Big(\frac{1}{4^j}\Big<\frac{\mathring{R}_0}{\varepsilon_0\delta}\Big>\Big)
\frac{1}{4^j}\partial_l\Big<\frac{\mathring{R}_0}{\varepsilon_0\delta}\Big>,
\eeno
thus
\beno
|\partial_l(\chi_j(t,x))|\leq \frac{C}{4^j}\frac{|\partial_l\mathring{R}_0|}{\varepsilon_0\delta}\leq C(\varepsilon_0, \mathring{R}_0, \delta), \quad |\partial_{ll}(\chi_j(t,x))|\leq C(\varepsilon_0, \mathring{R}_0, \delta).
\eeno
By the inequality
\beno
\|f\circ g\|_{C^L}\leq C_L\big(\|f\|_{C^L}\|g\|_{C^1}^L+ \|f\|_{C^1}\|g\|_{C^L}\big),
\eeno
we know that
\beno
\|\chi_j \|_{C^L_{t,x}}\leq C(\varepsilon_0, \mathring{R}_0, \delta, L).
\eeno
\end{proof}

\begin{Lemma}[Estimate on the amplitude]\label{e:estimate on the amplitde}
For $0\leq j\leq j_{max}$, we have
\beno
\|a_{(\xi,j)}\|_{L^\infty}&\leq & \sqrt{\rho_j} \leq 2^{j}\sqrt{\delta},\\
\|a_{(\xi,j)}\|_{C^L_{t,x}}&\leq& C(\mathring{R}_0, \delta,L) .
\eeno
\end{Lemma}
\begin{proof}
Recall that
\beno
a_{(\xi,j)}(t,x)=\sqrt{\rho_j}\chi_j(t,x)\gamma_{\xi}\Big({\rm Id}-\frac{\mathring{R}_0}{\rho_j}\Big).
\eeno
By Lemma \ref{e:$L^2$ and derivative estimate on cutoff function}, the estimate on $\|a_{(\xi,j)}\|_{L^\infty}$ and $\|a_{(\xi,j)}\|_{C^L_{t,x}}$ is obvious. In fact, noticing (\ref{e:estimate on number}), we have
\beno
\|a_{(\xi,j)}\|_{C^L_{t,x}}&=&\sqrt{\rho_j}\Big\|\chi_j\gamma_{\xi}\Big({\rm Id}-\frac{\mathring{R}_0}{\rho_j}\Big)\Big\|_{C^L_{t,x}}\\
&\leq& \sqrt{\rho_j}\|\chi_j\|_{C^L_{t,x}}\Big\|\gamma_{\xi}\Big({\rm Id}-\frac{\mathring{R}_0}{\rho_j}\Big)\Big\|_{C^L_{t,x}}\\
&\leq& C(\mathring{R}_0, \delta, L)\sqrt{\delta}2^{j_{max}}\leq C(\mathring{R}_0, \delta, L).
\eeno
\end{proof}

\begin{proposition}[Estimate on the perturbation]\label{e:estimate on the pertirbation}
For the velocity perturbation, we have the following bound: for every $t\in [0,1]$\\
1. $L^2$ estimate:
\beno
&&\|w^{p}_{1}(t)\|_{L^2}\leq \frac{M \sqrt{\delta}}{16}+ C(\mathring{R}_0,\delta)(\lambda_1 \sigma_1)^{-\frac12},\\
&&\|w^{c}_{1}(t)\|_{L^2}\leq  C(\mathring{R}_0,\delta)\sigma_1 r_1,\\[3pt]
 &&\|w^{t}_{1}(t)\|_{L^2}\leq  C(\mathring{R}_0,\delta)r_1\mu^{-1}_1,
\eeno
where $M$ is a universal constant.\\
2. $L^p$ estimate: for $p> 1$, there holds
\beno
&&\|w^{p}_{1}(t)\|_{L^p}\leq C(\mathring{R}_0, \delta, p)r_1^{1-\frac{2}{p}}, \\
 &&\|w^{c}_{1}(t)\|_{L^p}\leq C(\mathring{R}_0, \delta, p)\sigma_1 r_1^{2-\frac{2}{p}}, \\
 &&\|w^{t}_{1}(t)\|_{L^p}\leq C(\mathring{R}_0, \delta, p)  r_1^{2-\frac{2}{p}}\mu_1^{-1}.
\eeno
3. $W^{1,p}$ estimate: for $p> 1$, there holds
\beno
\|w^{p}_{1}(t)\|_{W^{1,p}}&\leq & C(\mathring{R}_0, \delta, p)\lambda_1 r_1^{1-\frac{2}{p}},\\ \|w^{c}_{1}(t)\|_{W^{1,p}}&\leq & C(\mathring{R}_0, \delta, p)\lambda_1\sigma_1 r_1^{2-\frac{2}{p}},\\
\|w^{t}_{1}(t)\|_{W^{1,p}}&\leq&  C(\mathring{R}_0, \delta, p)\mu_1^{-1}\lambda_1\sigma_1 r_1^{3-\frac{2}{p}}.
\eeno
In particular, we have
\beno
\|w^{p}_{1}(t)\|_{C^L}&\leq & C(\mathring{R}_0, \delta, L)\lambda_1^L r_1,\\
 \|w^{c}_{1}(t)\|_{C^L}&\leq & C(\mathring{R}_0, \delta,L)\lambda_1^L\sigma_1 r_1^2,\\
\|w^{t}_{1}(t)\|_{C^L}&\leq&  C(\mathring{R}_0, \delta, L)\mu_1^{-1}(\lambda_1\sigma_1 r_1)^L r_1^2.
\eeno
4. Time derivative estimate: for $p> 1$, there holds
\beno
\|\partial_tw^{p}_{1}(t)\|_{L^p}&\leq& C(\mathring{R}_0, \delta, p)\lambda_1\sigma_1 \mu_1 r_1^{2-\frac{2}{p}}.
\eeno
\end{proposition}
\begin{proof}
{\bf Step 1: $L^2$ estimate.}
Recall the definition (\ref{e:principle part}) of $w^p_{1}$, using the support property (\ref{e:basic property of cutoff function}) of cutoff function $\chi_i$,  Proposition \ref{e:estimate on building blocks}, Lemma \ref{e:estimate on the amplitde} and Lemma \ref{e:improved $L^p$ inequality}, we have
\beno
 &&\frac12 \sum_{j} \sum_{\xi \in \Lambda_{(j)}^+} \int_{{\rm T}^2}a^2_{(\xi,j)}\eta^2_{(\xi)}\Big|W_{(\xi)}+W_{(-\xi)}\Big|^2dx\\
&\leq& 2\sum_{j} \sum_{\xi \in \Lambda_{(j)}^+} \int_{{\rm T}^2}a^2_{(\xi,j)}\eta^2_{(\xi)}dx\\
&\leq& 2\sum_{j} \sum_{\xi \in \Lambda_{(j)}^+} \Big(\|a_{(\xi,j)}\|^2_{L^2}\|\eta_{(\xi)}\|^2_{L^2}+\frac{C}{\lambda_1 \sigma_1}\|a_{(\xi,j)}\|^2_{C^1}\|\eta_{(\xi)}\|^2_{L^2}\Big)\\
&\leq& 6\sum_{j} \Big(4^{j} \delta\int_{{\rm T}^2}\chi_j^2(t,x)dx+\frac{C(\mathring{R}_0,\delta)}{\lambda_1 \sigma_1}\Big).
\eeno
Moreover, for any $t\in [0,1]$, there hold
\beno
\sum_{j \geq 1} 4^{j} \delta\int_{{\rm T}^2}\chi_j^2(t,x)dx
&\leq & \delta \sum_{j \geq 1} 4^{j}\int_{\{x\in {\rm T}^2:4^{j-1}\varepsilon_0\delta\leq|\mathring{R}_0(t,x)|\leq 4^{j+1}\varepsilon_0\delta\}}dx\\
&\leq& \delta \sum_{j \geq 1}  4^{j}\int_{\{x\in {\rm T}^2:4^{j-1}\varepsilon_0\delta\leq|\mathring{R}_0(t,x)|\leq 4^{j+1}\varepsilon_0\delta\}}\frac{|\mathring{R}_0(t,x)|}{4^{j-1}\varepsilon_0\delta}dx\\
&\leq& 4\varepsilon_0^{-1} \sum_{j \geq 1}  \int_{\{x\in {\rm T}^2:4^{j-1}\delta\leq|\mathring{R}_0|\leq 4^{j+1}\delta\}}|\mathring{R}(t,x)|dx\\
&\leq& 4\varepsilon_0^{-1}\int_{{\rm T}^2}|\mathring{R}_0(t,x)|dx \leq \frac{\delta}{500}.
\eeno
Thus, we obtain
\ben\label{e:energy of first odd part}
 &&\frac12 \sum_{j \geq 1} \sum_{\xi \in \Lambda_{(j)}} \int_{{\rm T}^2}a^2_{(\xi,j)}\eta^2_{(\xi)}\Big|W_{(\xi)}+W_{(-\xi)}\Big|^2dx\leq \frac{\delta}{80}+\frac{C(\mathring{R}_0,\delta)}{\lambda_1 \sigma_1}.
\een
Finally, due to the definition of $\rho_0$ and estimate (\ref{e:energy difference estimate 1}), we deduce that
\beno
\frac12\int_{{\rm T}^2}\rho_0\chi_0^2\gamma_{\xi}^2\Big({\rm Id}-\frac{\mathring{R}}{\rho_0}\Big)\eta^2_{(\xi)}\Big|W_{(\xi)}+W_{(-\xi)}\Big|^2dx\leq C_0\delta,
\eeno
where $C_0$ is a universal constant.
Taking $M$ to be a universal constant, we obtain
\beno
\|w^{p}_{1}\|_{L^2}(t)\leq \frac{M \sqrt{\delta}}{16}+ C(\mathring{R},\delta)\frac{1}{\sqrt{\lambda_1 \sigma_1}}.
\eeno

From the definition (\ref{e:corretor part}) of $w^{c}_{1}$ , Proposition \ref{e:estimate on building blocks} and lemma \ref{e:estimate on the amplitde}, and noticing the fact (\ref{e:estimate on number})(we use this fact frequently below), we deduce that
\beno
\|w^{c}_{1}(t)\|_{L^2}&\leq& \frac{1}{\lambda_1}  \sum_{j} \sum_{\xi \in \Lambda_{(j)}} \big(\|a_{(\xi,j)}\|_{C^1_{x}}\|\eta_{(\xi)}\|_{L^2}+\|a_{(\xi,j)}\|_{L^\infty_{x}}\|\nabla\eta_{(\xi)}\|_{L^2_x}\big)\\
&\leq &\frac{1}{\lambda_1} \sum_{j} \sum_{\xi \in \Lambda_{(j)}} \big(C(\mathring{R}_0,\delta)+2^j \sqrt{\delta}\lambda_1\sigma_1 r_1\big)\\
&\leq& C(\mathring{R}_0,\delta)\sigma_1 r_1.
\eeno

By (\ref{e:defination of time corrector}), Proposition \ref{e:estimate on building blocks} and Lemma \ref{e:estimate on the amplitde}, we obtain
\beno
\|w^{t}_{1}(t)\|_{L^2}&\leq& \frac{1}{\mu_1}\sum_{j} \sum_{\xi \in \Lambda_{(j)}} \|a^2_{(\xi,j)}\eta^2_{(\xi)}\|_{L^2}\\
&\leq& \frac{1}{\mu_1} \sum_{j} \sum_{\xi \in \Lambda_{(j)}} \|a_{(\xi,j)}\|^2_{L^\infty}\|\eta_{(\xi)}\|^2_{L^4}\\
&\leq& \frac{1}{\mu_1}\sum_{j} \sum_{\xi \in \Lambda_{(j)}} 4^jr_1\leq  C(\mathring{R}_0,\delta)\frac{r_1}{\mu_1}.
\eeno
{\bf Step 2: $L^p$ estimate.}
By Lemma \ref{e:improved $L^p$ inequality}, proposition \ref{e:estimate on building blocks} and Lemma \ref{e:estimate on the amplitde}, we obtain
\beno
\|w^{p}_{1}(t)\|_{L^p}&\leq& \sum_{j} \sum_{\xi \in \Lambda_{(j)}}\Big( \|a_{(\xi,j)}(t)\|_{L^p}\|\eta_{(\xi)}(t)\|_{L^p}+C_p(\lambda_1\sigma_1)^{-\frac{1}{p}}
\|a_{(\xi,j)}(t)\|_{C^1_{x}}\|\eta_{(\xi)}(t)\|_{L^p}\Big)\\
&\leq & C(\mathring{R}_0, \delta, p)\sum_{j} \sum_{\xi \in \Lambda_{(j)}}2^j r_1^{1-\frac {2}{p}}\leq C(\mathring{R}_0, \delta, p)r_1^{1-\frac{2}{p}}.
\eeno
For $\|w^{c}_{1}(t)\|_{L^p}$ and $\|w^t_{1}\|_{L^p}$, there hold
\beno
\|w^{c}_{1}(t)\|_{L^p}
&&\leq  \frac{C}{\lambda_1}\sum_{j} \sum_{\xi \in \Lambda_{(j)}}\big(\|a_{(\xi,j)}\|_{C^1_{x}}\|\eta_{(\xi)}\|_{L^p}+\|a_{(\xi,j)}\|_{L^\infty_{x}}\|\nabla\eta_{(\xi)}\|_{L^p_x}\big)\\
&&\leq C(\mathring{R}_0, \delta, p) \sigma_1 r_1^{2-\frac{2}{p}}
\eeno
and
\beno
\|w^{t}_{1}(t)\|_{L^p}
&\leq & \frac{C}{\mu_1}\sum_{j} \sum_{\xi \in \Lambda_{(j)}}\|a_{(\xi,j)}\|^2_{L^\infty_{x}}\|\eta_{(\xi)}\|^2_{L^{2p}}
\\
&\leq &C(\mathring{R}_0, \delta, p)  r_1^{2-\frac{2}{p}}\mu_1^{-1}.
\eeno
{\bf Step 3: $W^{1,p}$ estimate.}
Recalling (\ref{e:principle part}), a direct computation gives that
\beno
\partial_l w^{p}_{1}&=&\frac{1}{\sqrt{2}}\sum_{j} \sum_{\xi \in \Lambda_{(j)}}\Big( \partial_l\big(a_{(\xi,j)} \eta_{(\xi)}\big)\big(W_{(\xi)}+W_{(-\xi)}\big)+a_{(\xi,j)} \eta_{(\xi)}\big(\partial_lW_{(\xi)}+\partial_lW_{(-\xi)}\big)\Big),
\eeno
thus by Lemma \ref{e:improved $L^p$ inequality}, Proposition \ref{e:estimate on building blocks} and Lemma \ref{e:estimate on the amplitde}, we have
\beno
\|\nabla w^{p}_{1} \|_{L^p}&\leq& \sum_{j} \sum_{\xi \in \Lambda_{(j)}}\Big( \|\nabla a_{(\xi,j)}\|_{L^p}\|\eta_{(\xi)}\|_{L^p}+\| a_{(\xi,j)}\|_{L^p}\|\nabla \eta_{(\xi)}\|_{L^p}\\
&&\quad +C_p(\lambda_1\sigma_1)^{-\frac{1}{p}}
\|a_{(\xi,j)}\|_{C^2_{x}}\|\eta_{(\xi)}\|_{W^{1,p}}\Big)\\
&&+\lambda_1\sum_{j} \sum_{\xi \in \Lambda_{(j)}}\Big( \| a_{(\xi,j)}\|_{L^p}\|\eta_{(\xi)}\|_{L^p}+C_p(\lambda_1\sigma_1)^{-\frac{1}{p}}
\|a_{(\xi)}\|_{C^1_{x}}\|\eta_{(\xi)}\|_{L^p}\Big)\\
&\leq & C(\mathring{R}_0, \delta, p)\lambda_1 r_1^{1-\frac{2}{p}},
\eeno
and
\beno
\|\nabla w^{p}_{1}\|_{L^\infty}\leq  C(\mathring{R}_0, \delta)\lambda_1 r_1.
\eeno
Recalling (\ref{e:corretor part}), there holds
\beno
\partial_l w^{c}_{1}:&=&-\frac{1}{\sqrt{2}}\sum_{j} \sum_{\xi \in \Lambda_{(j)}}\nabla^{\bot}\partial_l \big(a_{(\xi,j)} \eta_{(\xi)}\big)\frac{e^{i\lambda_1 \xi^{\bot}\cdot x}+e^{-i\lambda_1 \xi^{\bot}\cdot x}}{\lambda_1 }\nonumber\\
&&-\frac{1}{\sqrt{2}}\sum_{j} \sum_{\xi \in \Lambda_{(j)}}\nabla^{\bot} \big(a_{(\xi,j)} \eta_{(\xi)}\big)\partial_l\Big(\frac{e^{i\lambda_1 \xi^{\bot}\cdot x}+e^{-i\lambda_1 \xi^{\bot}\cdot x}}{\lambda_1 }\Big).
\eeno
Thus, by Lemma \ref{e:improved $L^p$ inequality}, Proposition \ref{e:estimate on building blocks} and Lemma \ref{e:estimate on the amplitde}, we get
\beno
\|\nabla w^{c}_{1}\|_{L^p}\leq  C(\mathring{R}_0, \delta, p)\lambda_1\sigma_1 r_1^{2-\frac{2}{p}}.
\eeno
Recalling (\ref{e:defination of time corrector}), we have
\beno
\partial_lw_{1}^{t}:=-\frac{1}{\mu_1}\sum_{j} \sum_{\xi \in \Lambda_{(j)}^+}P_H P_{\neq 0}\partial_l\big(a^2_{(\xi,j)}\eta^2_{(\xi)}\xi\big).
\eeno
Thus, by Lemma \ref{e:improved $L^p$ inequality}, Proposition \ref{e:estimate on building blocks} and Lemma \ref{e:estimate on the amplitde}, we get
\beno
\|\nabla w^{t}_{1}\|_{L^p}\leq  C(\mathring{R}_0, \delta, p)\mu_1^{-1}\lambda_1\sigma_1 r_1^{3-\frac{2}{p}}.
\eeno
The same argument gives that
\beno
\|w^{p}_{1}(t)\|_{C^L}\leq  C(\mathring{R}_0, \delta)\lambda^L r_1.
\eeno
The estimate for $\|w^{c}_{1}(t)\|_{C^L}, \|w^{t}_{1}(t)\|_{C^L}$ is similar, and we omit the detail here. \\

{\bf Step 4: Time derivative estimate.}
A direct computation gives that
\beno
\partial_t w^{p}_{1}&=&\frac{1}{\sqrt{2}}\sum_{j} \sum_{\xi \in \Lambda_{(j)}^+} \partial_t\big(a_{(\xi,j)} \eta_{(\xi)}\big)\big(W_{(\xi)}+W_{(-\xi)}\big),
\eeno
thus by Lemma \ref{e:improved $L^p$ inequality}, Proposition \ref{e:estimate on building blocks} and Lemma \ref{e:estimate on the amplitde}, we deduce that
\beno
\|\partial_t w^{p}_{1} \|_{L^p}&\leq& \sum_{j} \sum_{\xi \in \Lambda_{(j)}} \Big( \|\partial_t a_{(\xi,j)}\|_{L^p}\|\eta_{(\xi)}\|_{L^p}+\| a_{(\xi,j)}\|_{L^p}\|\partial_t \eta_{(\xi)}\|_{L^p}\\
&&\quad +C_p(\lambda_1\sigma_1)^{-\frac{1}{p}}
\|a_{(\xi,j)}\|_{C^2_{t,x}}\|\partial_t\eta_{(\xi)}\|_{L^p}\Big)
\leq  C(\mathring{R}_0, \delta, p)\lambda_1\sigma_1 \mu_1 r_1^{2-\frac{2}{p}}.
\eeno
\end{proof}
\begin{Corollary}\label{e:estimate on full perturbation}
For all $1< p< 2$, by taking $\lambda_1$ large enough, we have
\beno
\|w_{1}\|_{L^2}\leq \frac{M \sqrt{\delta}}{10}, \quad \quad \|w_{1}\|_{L^p}\leq C(\mathring{R}_0, \delta) r_1^{1-\frac{2}{p}}.
\eeno
\end{Corollary}
\begin{proof}
From the definition of $w_{1}$, by the $L^2$ estimate in Proposition \ref{e:estimate on the pertirbation}, we deduce that 
\beno
\|w_{1}\|_{L^2}\leq \frac{M\sqrt{\delta}}{16}+C(\mathring{R}_0, \delta)\big((\lambda_1\sigma_1)^{-\frac12}+ \sigma_1 r_1+ r_1 \mu_1^{-1}\big).
\eeno
Using the relationship of parameter (\ref{r:relationship of papameter}) and taking $\lambda_1$ large enough, we obtain
\beno
\|w_{1}\|_{L^2}\leq \frac{M\sqrt{\delta}}{10}.
\eeno
Similarly,  we get
\beno
\|w_{1}(t)\|_{L^p}\leq C(\mathring{R}_0, \delta, p)\Big(r_1^{1-\frac{2}{p}}+\sigma_1 r_1^{2-\frac{2}{p}}+ r_1^{2-\frac{2}{p}}\mu_1^{-1}\Big).
 \eeno
 Then, the parameter relationship gives
 \beno
\|w_{1}(t)\|_{L^p}\leq C(\mathring{R}_0, \delta, p)r_1^{1-\frac{2}{p}}.
 \eeno
\end{proof}


\section{Construction and Estimate on temperature perturbation}

After the construction of new velocity $v_{1}$, we construct new temperature $\theta_{1}$ as following.  Consider the transport-diffusion equation:
\begin{equation}
\begin{cases}
\partial_t \theta_{1}+ v_{1}\cdot\nabla  \theta_{1}-\triangle  \theta_{1}=0,\\[3pt]
\theta_{1}|_{t=0}=\theta^0,
\end{cases}
\end{equation}
where $\theta^0(x)$ is the function in Proposition \ref{p:iterative proposition}. From the standard theory, we know that there exists a unique solution $\theta_{1}\in C^\infty([0,1]\times {\rm T}^2)$ and it obeys the following estimates:
\beno
\|\theta_{1}\|_{L^\infty_{t,x}}\leq \|\theta^0\|_{L^\infty}, \quad \|\theta_{1}(t)\|^2_{L^2}+ 2\int_0^t \|\nabla \theta_{1}(s)\|^2_{L^2}ds=\|\theta^0\|^2_{L^2}
\eeno
and
\begin{equation}\label{e:transport difference equation}
\begin{cases}
\partial_t(\theta_{1}-\theta_0)+v_{1}\cdot\nabla (\theta_{1}-\theta_0)+(v_{1}-v_0)\cdot\nabla \theta_0-\triangle (\theta_{1}-\theta_0)=0,\\[3pt] (\theta_{1}-\theta_0)|_{t=0}=0.
\end{cases}
\end{equation}
Direct energy estimate gives that
\beno
\frac{d}{dt}\|\theta_{1}-\theta_0\|^2_{L^2}+2\|\nabla(\theta_{1}-\theta_0)\|^2_{L^2}&=&2\int_{{\rm T}^2}\theta_0(v_{1}-v_0)
\cdot\nabla(\theta_{1}-\theta_0)dx\\
&\leq &\|\nabla(\theta_{1}-\theta_0)\|^2_{L^2}+\parallel\theta_0\parallel_\infty \|v_{1}-v_0\|^2_{L^2}.
\eeno
which implies
\beno
\sup_{t\in [0,1]}\|\theta_{1}-\theta_0\|^2_{L^2}(t)+\int_0^1\|\nabla(\theta_{1}-\theta_0)\|^2_{L^2}(t)dt\leq \parallel\theta^0\parallel_\infty M^2\delta.
\eeno

\section{Reynold Stress: Construction and Estimate}

\subsection{Anti-divergence operator}
We first recall the anti-divergence operator:
\begin{Lemma}
There exists an operator $\mathcal{R}$ satisfying the following property:
\begin{itemize}
  \item For any $v\in C^\infty({\rm T}^2; R^2)$, $\mathcal{R}v(x)$ is a symmetric trace-free matrix for each $x\in {\rm T}^2$ and
         ${\rm div} \mathcal{R}v(x)= v(x)-\fint_{{\rm T}^2} v(x)dx.$
  \item The following estimates hold: $\||\nabla| \mathcal{R}\|_{L^p\rightarrow L^p}\leq C_p$,
$\|\mathcal{R}\|_{L^p\rightarrow L^p}\leq C_p, \quad \| \mathcal{R}\|_{C^0\rightarrow C^0}\leq C.$
\end{itemize}
\end{Lemma}
\begin{proof}
Let $u\in C_0^\infty({\rm T}^2)$ be a solution to
\beno
\triangle u=v-\fint_{{\rm T}^2}v(x)dx,
\eeno
where $C_0^\infty({\rm T}^2)=\{f\in C^\infty({\rm T}^2):\fint f(x)dx=0\}.$
Then set
\beno
\mathcal{R}v(x):= \nabla u+ (\nabla u)^T -({\rm div}u) {\rm Id}.
\eeno
Then $\mathcal{R}$ satisfies the above property.
\end{proof}

\subsection{Construction of new error $\mathring{R}_{1}$}
In this subsection, we define the new error $\mathring{R}_{1}$. We first compute the interaction {\bf $w^{p}_{1}\otimes w^{p}_{1}$} of principle perturbation. Recalling the definition (\ref{e:principle part}) of $w^{p}_{1}, $ we have
\beno
w^p_{1}\otimes w^p_{1} = T_{self} + T_{inter}, 
\eeno
where 
\beno
T_{self} &:=&\sum_{j} \sum_{\xi \in \Lambda_{(j)}^+} a^2_{(\xi,j)}\eta^2_{(\xi)}\Big( \xi\otimes\xi-\frac12 \xi\otimes\xi\big(e^{2i\lambda_1 \xi^{\bot}\cdot x}+e^{-2i\lambda_1 \xi^{\bot}\cdot x}\big)\Big)\\
&=&\sum_{j} \sum_{\xi \in \Lambda_{(j)}^+} a^2_{(\xi,j)}\xi\otimes\xi+\sum_{j} \sum_{\xi \in \Lambda_{(j)}^+}a^2_{(\xi,j)}\big(\eta^2_{(\xi)}-1\big)\xi\otimes\xi\\
&&-\frac12\sum_{j} \sum_{\xi \in \Lambda_{(j)}^+}a^2_{(\xi,j)}\eta^2_{(\xi)}\xi\otimes\xi\big(e^{2i\lambda_1 \xi^{\bot}\cdot x}+e^{-2i\lambda_1 \xi^{\bot}\cdot x}\big),
\eeno
and  $T_{inter} = w^p_{1}\otimes w^p_{1} - T_{self}$.
Recalling (\ref{e:amplitude definition}), we deduce
\beno
\sum_{j} \sum_{\xi \in \Lambda_{(j)}^+} a^2_{(\xi,j)}\xi\otimes\xi=\sum_{j} \sum_{\xi \in \Lambda_{(j)}^+} \rho_j\chi_j^2 \gamma^2_{\xi}\Big({\rm Id}-\frac{\mathring{R}_0}{\rho_j}\Big)\xi \otimes \xi.
\eeno
However, by Geometric Lemma \ref{l:geometric lemma},
\beno
\mathring{R}_0=\sum_{j}\chi_j^2 \mathring{R}_0
&=&\Big(\sum_{j}\rho_j\chi_j^2\Big) {\rm Id} -\sum_{j}\rho_j\chi_j^2 \Big({\rm Id}-\frac{\mathring{R}_0}{\rho_j}\Big)\\
&=&\Big(\sum_{j}\rho_j\chi_j^2 \Big){\rm Id} -\sum_{j}\rho_j\chi_j^2 \Big[\sum_{\xi \in \Lambda_{(j)}^+}\gamma_{\xi}^2\Big({\rm Id}-\frac{\mathring{R}_0}{\rho_j}\Big)\xi \otimes \xi \Big].
\eeno
Thus, there holds
\beno
\mathring{R}_0+\sum_{j} \sum_{\xi \in \Lambda_{(j)}^+}a^2_{(\xi,j)}\xi\otimes\xi=\Big(\sum_{j}\rho_j\chi_j^2\Big) {\rm Id}.
\eeno
Furthermore, by (\ref{e:transport structure}) and using the identity ${\rm div}f={\rm P}_{\neq 0}{\rm div}f$, there hold
\beno
&&{\rm div}\Big(\sum_{j} \sum_{\xi \in \Lambda_{(j)}^+}a^2_{(\xi,j)}\big(\eta^2_{(\xi)}-1\big)\xi\otimes\xi\Big)\\
&=&{\rm P}_{\neq 0}\Big(\sum_{j} \sum_{\xi \in \Lambda_{(j)}^+}\big(\eta^2_{(\xi)}-1\big)\xi\otimes\xi \nabla\big(a^2_{(\xi,j)}\big)\Big)+{\rm P}_{\neq 0}\Big(\sum_{j} \sum_{\xi \in \Lambda_{(j)}^+}a^2_{(\xi,j)}\xi\otimes\xi\nabla\big(\eta^2_{(\xi)}\big)\Big)\\
&=&{\rm P}_{\neq 0}\Big(\sum_{j} \sum_{\xi \in \Lambda_{(j)}^+}\big(\eta^2_{(\xi)}-1\big)\xi\otimes\xi \nabla\big(a^2_{(\xi,j)}\big)\Big)+{\rm P}_{\neq 0}\Big(\sum_{j} \sum_{\xi \in \Lambda_{(j)}^+}a^2_{(\xi,j)}\frac{\xi}{\mu_1}\partial_t\big(\eta^2_{(\xi)}\big)\Big)\\
&=&{\rm P}_{\neq 0}\Big(\sum_{j} \sum_{\xi \in \Lambda_{(j)}^+}\big(\eta^2_{(\xi)}-1\big)\xi\otimes\xi \nabla\big(a^2_{(\xi,j)}\big)\Big)+{\rm P}_{\neq 0}\Big(\frac{1}{\mu_1}\sum_{j} \sum_{\xi \in \Lambda_{(j)}^+}\partial_t\big(a^2_{(\xi,j)}\eta^2_{(\xi)}\xi\big)\Big)\\
&&-{\rm P}_{\neq 0}\Big(\frac{1}{\mu_1}\sum_{j} \sum_{\xi \in \Lambda_{(j)}^+}\partial_t\big(a^2_{(\xi,j)}\big)\eta^2_{(\xi)}\xi\Big)
\eeno
and
\beno
&&{\rm div}\Big(\frac12\sum_{j} \sum_{\xi \in \Lambda_{(j)}^+}a^2_{(\xi,j)}\eta^2_{(\xi)}\xi\otimes\xi\big(e^{2i\lambda_1 \xi^{\bot}\cdot x}+e^{-2i\lambda_1 \xi^{\bot}\cdot x}\big)\Big)\\
&=&\frac12\sum_{j} \sum_{\xi \in \Lambda_{(j)}^+}\xi\otimes\xi \nabla \big(a^2_{(\xi,j)}\eta^2_{(\xi)}\big)\big(e^{2i\lambda_1 \xi^{\bot}\cdot x}+e^{-2i\lambda_1 \xi^{\bot}\cdot x}\big).
\eeno
Using (\ref{e:basic property of cutoff function}), we have
\begin{align*}
T_{inter} &=\sum_{|j-j'| = 1} \sum_{\xi \in \Lambda_{(j)}^+, \xi' \in \Lambda_{(j')}^+} a_{(\xi,j)} a_{(\xi',j')}\eta_{(\xi)} \eta_{(\xi')}(W_{(\xi)} + W_{(-\xi)})\otimes (W_{(\xi')} + W_{(-\xi')})\\
& \quad + \sum_{j} \sum_{\xi, \xi' \in \Lambda_{(j)}^+, \xi \neq \xi' } a_{(\xi,j)} a_{(\xi',j)}\eta_{(\xi)} \eta_{(\xi')}(W_{(\xi)} + W_{(-\xi)})\otimes (W_{(\xi')} + W_{(-\xi')}).
\end{align*}
It follows from Proposition \ref{Prop:Beltrami-2D} that
\begin{align*}
\mathrm{div} ( W_{\xi,\xi'} \otimes W_{\xi,\xi'}) = 2\nabla \tilde{P}_{\xi, \xi'},
\end{align*}
where
\begin{align*}
W_{\xi,\xi'} &= W_{(\xi)} + W_{(-\xi)} + W_{(\xi')} + W_{(-\xi')},\\
\tilde{P}_{\xi, \xi'} &= \left|\xi\sin (\lambda \xi^{\perp}  \cdot x) + \xi' \sin (\lambda {\xi'}^{\perp}  \cdot x) \right|^2 + \left( \cos  (\lambda \xi^{\perp}  \cdot x) + \cos  (\lambda {\xi'}^{\perp} \cdot x)  \right)^2\\
&=2+2\cos (\lambda {\xi}^{\perp} \cdot x)\cos (\lambda {\xi'}^{\perp} \cdot x)
+2\xi \cdot \xi'\sin (\lambda \xi^{\perp}  \cdot x)\sin (\lambda {\xi'}^{\perp}  \cdot x).
\end{align*}
Set
\begin{align}\label{d:definition of pressue}
P_{\xi, \xi'} =2\cos (\lambda {\xi}^{\perp} \cdot x)\cos (\lambda {\xi'}^{\perp} \cdot x)
+2\xi \cdot \xi'\sin (\lambda \xi^{\perp}  \cdot x)\sin (\lambda {\xi'}^{\perp}  \cdot x),
\end{align}
then
\begin{align*}
\mathrm{div} ( W_{\xi,\xi'} \otimes W_{\xi,\xi'}) = 2\nabla P_{\xi, \xi'}.
\end{align*}
Thus, we obtain
\begin{align*}
 \mathrm{div} T_{inter} =  &\sum_{|j-j'| = 1} \sum_{\xi \in \Lambda_{(j)}^+, \xi' \in \Lambda_{(j')}^+}
\Big(  (W_{(\xi)} + W_{(-\xi)})\otimes (W_{(\xi')} + W_{(-\xi')}) \nabla  (a_{(\xi,j)} a_{(\xi',j')}\eta_{(\xi)} \eta_{(\xi')}) \\[5pt]
&+ \nabla \big(a_{(\xi,j)} a_{(\xi',j')}\eta_{(\xi)} \eta_{(\xi')}P_{\xi,\xi'}\big) - P_{\xi,\xi'} \nabla  ( a_{(\xi,j)} a_{(\xi',j')}\eta_{(\xi)} \eta_{(\xi')})\Big)\\[5pt]
+&\sum_{j} \sum_{\xi, \xi' \in \Lambda_{(j)}^+, \xi \neq \xi' }\Big(  (W_{(\xi)} + W_{(-\xi)})\otimes (W_{(\xi')} + W_{(-\xi')}) \nabla  (a_{(\xi,j)} a_{(\xi',j)}\eta_{(\xi)} \eta_{(\xi')}) \\[5pt]
&+ \nabla \big(a_{(\xi,j)} a_{(\xi',j)}\eta_{(\xi)} \eta_{(\xi')}P_{\xi,\xi'}\big) - P_{\xi,\xi'} \nabla  ( a_{(\xi,j)} a_{(\xi',j)}\eta_{(\xi)} \eta_{(\xi')})\Big).
\end{align*}
Finally, by combining the definition (\ref{e:defination of time corrector}) of $w^t_{1}$, we obtain
\ben\label{e:computation of main term}
&&{\rm div} T_{self} + {\rm div}  T_{inter} +{\rm div}\mathring{R}_0+\partial_t w_{1}^{t} = -\nabla(p_1 - p_0) + T_{1,{\rm osc}}, 
\een
where we define the new pressure $p_{1}$ such that
\beno
p_{1}-p_0&=-\Big(&\sum_{j}\rho_j\chi_j^2+\frac{1}{\mu_1}\sum_{j} \sum_{\xi \in \Lambda_{(j)}^+}\triangle^{-1}{\rm div}\partial_t\big(a^2_{(\xi,j)}\eta^2_{(\xi)}\xi\big) \\
& &+ \sum_{|j-j'| = 1} \sum_{\xi \in \Lambda_{(j)}^+, \xi' \in \Lambda_{(j')}^+} a_{(\xi,j)} a_{(\xi',j')}\eta_{(\xi)} \eta_{(\xi')}P_{\xi,\xi'}\\
 &&+  \sum_{j} \sum_{\xi, \xi' \in \Lambda_{(j)}^+, \xi \neq \xi' }a_{(\xi,j)} a_{(\xi',j)}\eta_{(\xi)} \eta_{(\xi')}P_{\xi,\xi'}\Big),
\eeno
and oscillatory term
\ben\label{e:oscillatory term definition}
T_{1,{\rm osc}}:&=&{\rm P}_{\neq 0}\Big(\sum_{j} \sum_{\xi \in \Lambda_{(j)}^+}\big(\eta^2_{(\xi)}-1\big)\xi\otimes\xi \nabla\big(a^2_{(\xi,j)}\big)\Big)\nonumber\\
&&-{\rm P}_{\neq 0}\Big(\frac{1}{\mu_1}\sum_{j} \sum_{\xi \in \Lambda_{(j)}^+}\partial_t\big(a^2_{(\xi,j)}\big)\eta^2_{(\xi)}\xi\Big)\nonumber\\
&&-\frac12\sum_{j} \sum_{\xi \in \Lambda_{(j)}^+}\xi\otimes\xi \nabla \big(a^2_{(\xi,j)}\eta^2_{(\xi)}\big)\big(e^{2i\lambda_1 \xi^{\bot}\cdot x}+e^{-2i\lambda_1 \xi^{\bot}\cdot x}\big)\nonumber\\
&&+{\rm P}_{\neq 0} \Big[ \sum_{|j-j'| = 1} \sum_{\xi \in \Lambda_{(j)}^+, \xi' \in \Lambda_{(j')}^+} \Big(- P_{\xi,\xi'} \nabla  ( a_{(\xi,j)} a_{(\xi',j')}\eta_{(\xi)} \eta_{(\xi')})  \nonumber\\
&& \quad \quad + (W_{(\xi)} + W_{(-\xi)})\otimes (W_{(\xi')} + W_{(-\xi')}) \nabla  \left( a_{(\xi,j)} a_{(\xi',j')}\eta_{(\xi)} \eta_{(\xi')}\right)\Big) \Big]\nonumber\\[5pt]
&&+{\rm P}_{\neq 0} \Big[ \sum_{j} \sum_{\xi, \xi' \in \Lambda_{(j)}^+, \xi \neq \xi' } \Big(- P_{\xi,\xi'} \nabla  ( a_{(\xi,j)} a_{(\xi',j)}\eta_{(\xi)} \eta_{(\xi')})  \nonumber\\
&& \quad \quad + (W_{(\xi)} + W_{(-\xi)})\otimes (W_{(\xi')} + W_{(-\xi')}) \nabla  \left( a_{(\xi,j)} a_{(\xi',j)}\eta_{(\xi)} \eta_{(\xi')}\right)\Big) \Big].
\een
From (\ref{e:computation of main term}), we know that $\fint_{{\rm T}^2}T_{1,{\rm osc}}(t,x)dx=0.$

 After the computation of interaction of principle perturbation, we define $\mathring{R}_{1}$ as follows:
\beno
\mathring{R}_{1}&=&\underbrace{\mathcal{R}\big(\partial_t (w^{p}_{1}+w^{c}_{1})\big)+ \mathcal{R}\big({\rm div}(v_0\otimes w_{1}+ w_{1}\otimes v_0)\big)+\mathcal{R}((-\Delta)^\alpha w_{1} )}_{R_{linear}}\\[3pt]
&&+\underbrace{\mathcal{R}\big({\rm div}(w^{p}_{1}\otimes (w^{c}_{1}+w^{t}_{1})+ (w^{c}_{1}+w^{t}_{1})\otimes w_{1})\big)}_{R_{cor}}\\[3pt]
&&+\underbrace{\mathcal{R}\big(T_{1,{\rm osc}}\big)}_{R_{osc}}-\underbrace{\mathcal{R}((\theta_{1}-\theta_{0}) e_{2})}_{R_{tem}}. 
\eeno
Since $\fint_{{\rm T}^2}\theta^0(x)dx=0,$ thus $\fint_{{\rm T}^2}\theta_{1}(t,x)dx=0$ for every $t$, hence  $\fint_{{\rm T}^2}(\theta_{1}-\theta_0)(t,x)dx=0$ for every $t$. Recalling the construction of $v_{1}, p_{1}, \theta_{1}, \mathring{R}_{1}$ and (\ref{e:computation of main term}), a direct computation gives that

\ben\label{e:new equation of new function}
{\rm div}\mathring{R}_{1}&=&\partial_t (w^{p}_{1}+w^{c}_{1})+{\rm div}(v_0\otimes w_{1}+ w_{1}\otimes v_0)+(-\Delta)^\alpha w_{1}\nonumber\\[3pt]
&&+{\rm div}(w^{p}_{1}\otimes (w^{c}_{1}+w^{t}_{1})+ (w^{c}_{1}+w^{t}_{1})\otimes w_{1})\big)\nonumber\\[3pt]
&&+T_{1,{\rm osc}}-(\theta_{1}-\theta_{0}) e_{2}\nonumber\\
&=&\partial_t (w^{p}_{1}+w^{c}_{1})+{\rm div}(v_0\otimes w_{1}+ w_{1}\otimes v_0)+(-\Delta)^\alpha w_{1}\nonumber\\[3pt]
&&+{\rm div}(w^{p}_{1}\otimes (w^{c}_{1}+w^{t}_{1})+ (w^{c}_{1}+w^{t}_{1})\otimes w_{1})\big)\nonumber\\[3pt]
&&+{\rm div}\big(w^p_{1}\otimes w^p_{1}\big)+{\rm div}\mathring{R}_0+\partial_t w_{1}^{t}+\nabla(p_{1}-p_0)-(\theta_{1}-\theta_{0}) e_{2}\nonumber\\[3pt]
&=&\partial_t w_{1}+{\rm div}(v_0\otimes w_{1}+ w_{1}\otimes v_0)+(-\Delta)^\alpha w_{1}
+{\rm div}\big(w_{1}\otimes w_{1}\big)+{\rm div}\mathring{R}_0\nonumber\\[3pt]
&&+\nabla(p_{1}-p_0)-(\theta_{1}-\theta_{0}) e_{2}\nonumber\\[3pt]
&=& \partial_t v_{1}+{\rm div}\big(v_{1}\otimes v_{1}\big)+\nabla p_{1}+(-\Delta)^\alpha v_{1}-\theta_{1} e_{2}.
\een
Thus, the new function $(v_{1}, p_{1}, \theta_{1}, \mathring{R}_{1})$ satisfies Boussinesq-Reynold equation (\ref{e:Boussinesq-Reynold equation}). Next, we prove that the error $\|\mathring{R}_{1}\|_{L^\infty_t L^1_x}$ is very small.

\subsection{Estimate on $\mathring{R}_{1}$}

In this subsection, we estimate $\mathring{R}_{1}$. We deal with it term by term. \\
\subsubsection{ Estimate on linear term $R_{linear}$:} For $1< p < 2$,
\beno
\|\mathcal{R}\big(\partial_t (w^{p}_{1}+w^{c}_{1})\big)\|_{L^p}&\leq& C(\mathring{R}_0,\delta, p)\sigma_1 \mu_1 r_1^{2-\frac{2}{p}},\\
\|\mathcal{R}((-\Delta)^\alpha w_{1})\|_{L^p}&\leq & C(\mathring{R}_0,\delta, p)\lambda_1^{2\alpha-1}r_1^{1-\frac{2}{p}},\\
\|\mathcal{R}\big({\rm div}(v_0\otimes w_{1}+ w_{1}\otimes v_0)\big)\|_{L^p}&\leq& C(\mathring{R}_0,v_0, \delta, p) r_1^{1-\frac{2}{p}}.
\eeno
\begin{proof}
Due to the $L^p$ estimate, $W^{1,p}$ estimate and  time derivative estimate in Proposition \ref{e:estimate on the pertirbation}, we deduce
\beno
&&\|\mathcal{R}\big(\partial_t (w^{p}_{1}+w^{c}_{1})\big)\|_{L^p}=\frac{1}{\lambda_1}\|\mathcal{R}\nabla^{\bot}(\partial_t w^{p}_{1}\big)\|_{L^p}\leq \frac{C_p}{\lambda_1}\|\partial_t w^{p}_{1}\|_{L^p}\leq C(\mathring{R}_0,\delta, p) \sigma_1 \mu_1 r_1^{2-\frac{2}{p}},\\[5pt]
&&\|\mathcal{R}((-\Delta)^\alpha w_{1})\|_{L^p}\leq C \|\mathcal{R}|\nabla|^{2\alpha}(w_{1})\|_{L^p} \leq C(\mathring{R}_0,\delta, p)\lambda_1^{2\alpha-1}r_1^{1-\frac{2}{p}},\\[5pt]
&&\|\mathcal{R}\big({\rm div}(v_0\otimes w_{1}+ w_{1}\otimes v_0)\big)\|_{L^p} \leq C_p\|v_0\otimes w_{1}+ w_{1}\otimes v_0\|_{L^p}\leq C(\mathring{R}_0, v_0, \delta, p) r_1^{1-\frac{2}{p}}.
\eeno
\end{proof}

\subsubsection{ Estimate on corrector term $R_{cor}$:}
For $1<p< 2$,
\beno
\|\mathcal{R}\big({\rm div}(w^{p}_{1}\otimes (w^{c}_{1}+w^{t}_{1})+ (w^{c}_{1}+w^{t}_{1})\otimes w_{1})\big)\|_{L^p}\leq C(\mathring{R}_0,\delta,p)(\sigma_1+\mu_1^{-1})r_1^{3(1-\frac{1}{p})}.
\eeno
\begin{proof}
By Proposition \ref{e:estimate on the pertirbation} and parameter relationship (\ref{r:relationship of papameter}), we get
\beno
&&\|\mathcal{R}\big({\rm div}(w^{p}_{1}\otimes (w^{c}_{1}+w^{t}_{1})+ (w^{c}_{1}+w^{t}_{1})\otimes w_{1})\big)\|_{L^p}\\[3pt]
&\leq & C_p\|w^{p}_{1}\otimes (w^{c}_{1}+w^{t}_{1})+ (w^{c}_{1}+w^{t}_{1})\otimes w_{1}\|_{L^p}\\[3pt]
&\leq& C_p\Big(\|w^{p}_{1}\otimes (w^{c}_{1}+w^{t}_{1})\|_{L^1}^{\frac{1}{p}}\|w^{p}_{1}\otimes (w^{c}_{1}+w^{t}_{1})\|_{L^\infty}^{1-\frac{1}{p}}\\[3pt]
&&+ \|(w^{c}_{1}+w^{t}_{1})\otimes w_{1}\|_{L^1}^{\frac{1}{p}} \|(w^{c}_{1}+w^{t}_{1})\otimes w_{1}\|_{L^\infty}^{1-\frac{1}{p}}\Big)\\
&\leq& C(\mathring{R}_0, \delta, p) (\sigma_1+\mu_1^{-1})r_1^{3(1-\frac{1}{p})}.
\eeno
\end{proof}

\subsubsection{Estimate on temperature term $R_{tem}$:} For $1< p <2$,
\beno
\|\mathcal{R}\big((\theta_{1}-\theta_{0}) e_{2}\big)\|_p\leq C(\theta_0, \mathring{R}_0, \delta, p)r_1^{1-\frac{2}{p}}.
\eeno
\begin{proof}
For $1< p<2$, we have
\beno
\|\mathcal{R}\big((\theta_{1}-\theta_{0}) e_{2}\big)\|_p\leq C_p \|\theta_{1}-\theta_{0}\|_p.
\eeno
From the equation (\ref{e:transport difference equation}), we try $L^p$ estimates: Multiplying $|\theta_{1}-\theta_{0}|^{p-2}(\theta_{1}-\theta_{0})$, we arrive at
\beno
&&\frac{1}{p}\frac{d}{dt}\|\theta_{1}-\theta_{0}\|^p_p+ \int_{{\rm T}^2}\nabla(\theta_{1}-\theta_{0})\cdot \nabla \big(|\theta_{1}-\theta_{0}|^{p-2}(\theta_{1}-\theta_{0})\big)dx\\
&\leq & \|v_{1}-v_{0}\|_p \|\nabla \theta_0\|_\infty \|\theta_{1}-\theta_{0}\|^{p-1}_p.
\eeno
A direct computation gives
\beno
\int_{{\rm T}^2}\nabla(\theta_{1}-\theta_{0})\cdot \nabla \big(|\theta_{1}-\theta_{0}|^{p-2}(\theta_{1}-\theta_{0})\big)dx=(p-1)\int_{{\rm T}^2}|\nabla(\theta_{1}-\theta_0)|^2 |\theta_{1}-\theta_0|^{p-2}.
\eeno
Thus, we obtain
\beno
\|\theta_{1}-\theta_{0}\|_p(t)&\leq & \int_0^t \|\nabla \theta_0\|_\infty(s)\|v_{1}-v_{0}\|_p(s)ds\\
&\leq& \int_0^1 \|\nabla \theta_0\|_\infty(s)\|v_{1}-v_{0}\|_p(s)ds,
\eeno
then by Corollary \ref{e:estimate on full perturbation}
\beno
\sup_{t\in [0,1]}\|\theta_{1}-\theta_{0}\|_p\leq \sup_{t\in [0,1]}\|\nabla \theta_0\|_{L^\infty_{x}} \sup_{t\in [0,1]} \|v_{1}-v_{0}\|_p\leq C(\theta_0, \mathring{R}_0, \delta, p)r_1^{1-\frac{2}{p}}.
\eeno
\end{proof}

\subsubsection{Estimate on the oscillatory term $\mathcal{R}_{osc}$:} For every $1< p< 2$,
\beno
\|\mathcal{R}_{osc}\|_{L^p}\leq C(\mathring{R}_0, \delta, p)(\lambda_1\sigma_1)^{-1} r_1^{2-\frac{2}{p}}.
\eeno
\begin{proof}
Recall (\ref{e:oscillatory term definition}). 
By Lemma \ref{e:oscillatory lemma} and Proposition \ref{e:estimate on building blocks}
\beno
&&\Big\|\mathcal{R}{\rm P}_{\neq 0}\Big(\sum_{j} \sum_{\xi \in \Lambda_{(j)}^+}\big(\eta^2_{(\xi)}-1\big)\xi\otimes\xi \nabla\big(a^2_{(\xi,j)}\big)\Big)\Big\|_{L^p}\\
&\leq &\sum_{j} \sum_{\xi \in \Lambda_{(j)}^+}\Big\||\nabla|^{-1}{\rm P}_{\neq 0}\Big(\big(\eta^2_{(\xi)}-1\big)\xi\otimes\xi \nabla\big(a^2_{(\xi,j)}\big)\Big)\Big\|_{L^p}\\
&\leq & C(\mathring{R}_0, \delta, p)\Big(1+\|a^2_{(\xi,j)}\|_{C^{3}}\Big)\frac{\|\eta^2_{(\xi)}-1\|_{L^p}}{\lambda_1 \sigma_1}\leq C(\mathring{R}_0, \delta, p)\frac{r_1^{2-\frac{2}{p}}}{\lambda_1\sigma_1}.
\eeno

Similarly, we have
\beno
&&\Big\|\mathcal{R}\Big(\sum_{j} \sum_{\xi \in \Lambda_{(j)}^+}\xi\otimes\xi \nabla \big(a^2_{(\xi,j)}\eta^2_{(\xi)}\big)\big(e^{2i\lambda_1 \xi^{\bot}\cdot x}+e^{-2i\lambda_1 \xi^{\bot}\cdot x}\big)\Big) \Big\|_{L^p}\\
&\leq & \Big\||\nabla|^{-1}\Big(\sum_{j} \sum_{\xi \in \Lambda_{(j)}^+}\xi\otimes\xi \nabla \big(a^2_{(\xi,j)}\big)\eta^2_{(\xi)}\big(e^{2i\lambda_1 \xi^{\bot}\cdot x}+e^{-2i\lambda_1 \xi^{\bot}\cdot x}\big)\Big) \Big\|_{L^p}\\
&&+\Big\||\nabla|^{-1}\Big(\sum_{j} \sum_{\xi \in \Lambda_{(j)}^+}\xi\otimes\xi a^2_{(\xi,j)}\nabla\big(\eta^2_{(\xi)}\big)\big(e^{2i\lambda_1 \xi^{\bot}\cdot x}+e^{-2i\lambda_1 \xi^{\bot}\cdot x}\big)\Big) \Big\|_{L^p}\\
&\leq &C(\mathring{R}_0, \delta, p)\Big(1+\|a^2_{(\xi,j)}\|_{C^{3}}\Big)\frac{\|\eta^2_{(\xi)}\|_{L^p}}{\lambda_1 }+C(\mathring{R}_0, \delta, p)\Big(1+\|a^2_{(\xi,j)}\|_{C^{3}}\Big)\frac{\|\nabla(\eta^2_{(\xi)})\|_{L^p}}{\lambda_1 }\\
&\leq& C(\mathring{R}_0, \delta, p)\sigma_1 r_1^{3-\frac{2}{p}}.
\eeno
Recalling Remark \ref{r:remark about universal constant} and (\ref{d:definition of pressue}),  we have
\begin{align*}
& \quad \Big\|\mathcal{R}{\rm P}_{\neq 0} \Big( \sum_{|j-j'| = 1} \sum_{\xi \in \Lambda_{(j)}^+, \xi' \in \Lambda_{(j')}^+}  P_{\xi,\xi'} \nabla  \left( a_{(\xi,j)} a_{(\xi',j')}\eta_{(\xi)} \eta_{(\xi')}\right)\Big)\Big\|_{L^p}
 \leq C(\mathring{R}_0, \delta, p)\sigma_1 r_1^{3-\frac{2}{p}},\\
&\quad \Big\|\mathcal{R}{\rm P}_{\neq 0}\Big( \sum_{|j-j'| = 1} \sum_{\xi \in \Lambda_{(j)}^+, \xi' \in \Lambda_{(j')}^+} (W_{(\xi)} + W_{(-\xi)})\otimes (W_{(\xi')} + W_{(-\xi')})\nabla \left( a_{(\xi,j)} a_{(\xi',j')}\eta_{(\xi)} \eta_{(\xi')}\right) \Big) \Big\|_{L^p}\\
& \leq C(\mathring{R}_0, \delta, p)\sigma_1 r_1^{3-\frac{2}{p}}.
\end{align*}
Similarly,
\begin{align*}
& \quad \Big\|\mathcal{R}{\rm P}_{\neq 0} \Big( \sum_{j} \sum_{\xi, \xi' \in \Lambda_{(j)}^+, \xi \neq \xi' }   P_{\xi,\xi'} \nabla  \left( a_{(\xi,j)} a_{(\xi',j)}\eta_{(\xi)} \eta_{(\xi')}\right)\Big)\Big\|_{L^p}
 \leq C(\mathring{R}_0, \delta, p)\sigma_1 r_1^{3-\frac{2}{p}},\\
&\quad \Big\|\mathcal{R}{\rm P}_{\neq 0}  \Big(\sum_{j} \sum_{\xi, \xi' \in \Lambda_{(j)}^+, \xi \neq \xi' }  (W_{(\xi)} + W_{(-\xi)})\otimes (W_{(\xi')} + W_{(-\xi')})\nabla \left( a_{(\xi,j)} a_{(\xi',j)}\eta_{(\xi)} \eta_{(\xi')}\right) \Big) \Big\|_{L^p}\\
& \leq C(\mathring{R}_0, \delta, p)\sigma_1 r_1^{3-\frac{2}{p}}.
\end{align*}
By Proposition \ref{e:estimate on building blocks} and Lemma \ref{e:estimate on the amplitde}, there hold
\beno
&&\Big\|\frac{1}{\mu_1}\mathcal{R}{\rm P}_{\neq 0}\Big(\sum_{j} \sum_{\xi \in \Lambda_{(j)}^+}\partial_t\big(a^2_{(\xi,j)}\big)\eta^2_{(\xi)}\xi\Big)\Big\|_{L^p}\\
&\leq& \frac{C_p}{\mu_1} \sum_{j} \sum_{\xi \in \Lambda_{(j)}^+}\|\partial_t\big(a^2_{(\xi,j)}\big)\|_{L^\infty} \|\eta^2_{(\xi)}\|_{L^p}
\leq C(\mathring{R}_0, \delta, p)r_1^{2-\frac{2}{p}}\mu_1^{-1}.
\eeno
Summing the parts and using the parameter relationship (\ref{r:relationship of papameter}), we complete the proof.
\end{proof}


Finally, collecting these term together and noticing that $\alpha < 1$, we obtain the estimate on the error term $\mathring{R}_{1}$:
\beno
&&\|\mathring{R}_{1}\|_{L^p}\\
&\leq&  C(\mathring{R}_0,\delta, v_0, \theta_0, p)\Big((\lambda_1\sigma_1)^{-1} r_1^{2-\frac{2}{p}} +r_1^{1-\frac{2}{p}}+\sigma_1\mu_1 r_1^{2-\frac{2}{p}}+\sigma_1 r_1^{3-\frac{2}{p}}+\mu_1^{-1}r_1^{3(1-\frac{1}{p})}+\lambda_1^{2\alpha -1}r_1^{1-\frac{2}{p}}\Big).
\eeno
Using the parameter relationship (\ref{r:relationship of papameter}), taking $1< p< \frac{2\alpha}{3\alpha -1}$ and noticing $\frac 12 \leq \alpha < 1$, we obtain
\beno
\|\mathring{R}_{1}\|_{L^p}
\leq  C(\mathring{R}_0,\delta, v_0, \theta_0, p)\lambda_1^{3\alpha-1-\frac{2\alpha}{p}}.
\eeno

{\bf In summary:} we have constructed smooth function $(v_{1}, p_{1}, \theta_{1}, R_{1})\in C^\infty([0,1]\times {\rm T}^2, R^2\times R\times R\times S^{2\times 2})$, they satisfies Boussinesq-Reynold equation (\ref{e:Boussinesq-Reynold equation}).
By taking $\lambda_1$ large enough, the following estimate hold: for any $t\in [0,1]$
\beno
\|v_{1}-v_0\|_{L^2}(t)&\leq& \frac{M \sqrt{\delta}}{10}, \quad \|\theta_{1}\|_{L^{\infty}_{t,x}}\leq\|\theta^0\|_{\infty},\\
\|\theta_{1}\|_{L^2}(t)&+& 2\int_0^t \|\nabla \theta_{1}\|_2^2(s) ds= \|\theta^0\|_{L^2}(t),\\
\|\theta_{1}-\theta_0\|^2_{L^2}(t)&+& \int_0^t \|\nabla (\theta_{1}-\theta_0)\|_2^2(s) ds\leq  \|\theta^0\|_\infty M^2 \delta,\\
\|\mathring{R}_{1}\|_{L^\infty_t L^1_x}&\leq& C(\mathring{R}_0,\delta, v_0, \theta_0, p)\lambda_1^{3\alpha-1-\frac{2\alpha}{p}}.
\eeno

\section{Proof of main Proposition}
In this section, we give a proof of Proposition \ref{p:iterative proposition} by combining the above construction and estimate.

Noticing that $3\alpha-1-\frac{2\alpha}{p}< 1$ for $1< p<\frac{2\alpha}{3\alpha-1}$, thus we first take $\lambda_1$ to be a integer, large enough such that
\beno
C(\mathring{R}_0,\delta, v_0, \theta_0, p)\lambda_1^{3\alpha-1-\frac{2\alpha}{p}}\leq \frac{\varepsilon_0\delta}{60000},
\eeno
Thus, there holds
\beno
\|\mathring{R}_1\|_{L^\infty_t L^1_x}\leq \frac{\varepsilon_0\delta}{20000}.
\eeno

To complete the proof of Proposition \ref{p:iterative proposition}, we only need to estimate the energy difference between $e(t)$ and $\int_{{\rm T}^2}|v_1(t,x)|^2dx.$

Direct computation gives that
\beno
\int_{{\rm T}^2}|v_1(t,x)|^2dx&= &\int_{{\rm T}^2}|v_0(t,x)+w_{1}(t,x)|^2dx\\
&=&\int_{{\rm T}^2}\Big(|v_0(t,x)|^2+|w_{1}(t,x)|^2\Big)dx + 2\int_{{\rm T}^2}v_0(t,x)\cdot w_{1}(t,x)dx.
\eeno
From the definition of $w_{1}^p$, we deduce
\beno
&&\int_{{\rm T}^2}|w_{1}^p(t,x)|^2dx = \underbrace{\frac{1}{2}\int_{{\rm  T}^2} \sum_{j} \sum_{\xi \in \Lambda_{(j)}^+} a_{(\xi,j)}^2 \eta_{(\xi)}^2\Big|W_{(\xi)}+W_{(-\xi)}\Big|^2dx}_{I}\\
&&+ \underbrace{ \frac{1}{2}\sum_{|j-j'| = 1} \sum_{\xi \in \Lambda_{(j)}^+, \xi' \in \Lambda_{(j')}^+} \int_{{\rm T}^2} a_{(\xi,j)} a_{(\xi',j')}\eta_{(\xi)} \eta_{(\xi')}\big(W_{(\xi)}+W_{(-\xi)}\big)\cdot \big(W_{(\xi')}+W_{(-\xi')}\big)dx}_{II}\\
&&+ \underbrace{ \frac{1}{2} \sum_{j} \sum_{\xi, \xi' \in \Lambda_{(j)}^+, \xi \neq \xi' } \int_{{\rm T}^2} a_{(\xi,j)} a_{(\xi',j)}\eta_{(\xi)} \eta_{(\xi')}\big(W_{(\xi)}+W_{(-\xi)}\big)\cdot \big(W_{(\xi')}+W_{(-\xi')}\big)dx}_{II}.
\eeno
Recalling (\ref{e:principle part}) and (\ref{e:basic property of cutoff function}), it's easy to deduce that
\beno
I&=&\sum_{\xi \in \Lambda_0^+}\int_{{\rm T}^2}\rho_0\chi_0^2\gamma_{\xi}^2\Big({\rm Id}-\frac{\mathring{R}_0}{\rho_0}\Big)dx + \sum_{\xi \in \Lambda_0^+}\int_{{\rm T}^2}\rho_0\chi_0^2\gamma_{\xi}^2\Big({\rm Id}-\frac{\mathring{R}_0}{\rho_0}\Big)\big(\eta^2_{(\xi)}-1\big)dx\\
&&+\frac12 \sum_{\xi \in \Lambda_0^+}\int_{{\rm T}^2}\rho_0\chi_0^2\gamma_{\xi}^2\Big({\rm Id}-\frac{\mathring{R}_0}{\rho_0}\Big)\eta^2_{(\xi)}\Big(-e^{2i \lambda_1 \xi^{\bot}\cdot x}-e^{-2i\lambda_1 \xi^{\bot}\cdot x}\Big)dx\\
&&+ \frac{1}{2}\int_{{\rm  T}^2} \sum_{j \geq 1} \sum_{\xi \in \Lambda_{(j)}^+} a_{(\xi,j)}^2 \eta_{(\xi)}^2\Big|W_{(\xi)}+W_{(-\xi)}\Big|^2dx.
\eeno
Recalling (\ref{e:identity in Geomertic constant}), there holds
\beno
\sum_{\xi \in \Lambda_0^+}\int_{{\rm T}^2}\rho_0\chi_0^2\gamma^2_{\xi}\Big({\rm Id}-\frac{\mathring{R}_0}{\rho_0}\Big)dx=2\rho_0 \int_{{\rm T}^2}\chi_0^2dx.
\eeno
Set
\beno
E(t)_{err}:&=&
\underbrace{2\int_{{\rm T}^2}v_0(t,x)\cdot w_{1}dx}_{(1)} +\underbrace{\int_{{\rm T}^2}\Big(2 w_{1}^p\cdot \big(w^c_{1}+w^t_{1}\big)+\big|w^c_{1}+w^t_{1}\big|^2\Big)dx + II}_{(2)}\\
&&+\underbrace{ \sum_{\xi \in \Lambda_0^+}\int_{{\rm T}^2}\rho_0\chi_0^2\gamma_{\xi}^2\Big({\rm Id}-\frac{\mathring{R}_0}{\rho_0}\Big)\big(\eta^2_{(\xi)}-1\big)dx}_{(3)}\\
&&+\underbrace{\frac12 \sum_{\xi \in \Lambda_0^+}\int_{{\rm T}^2}\rho_0\chi_0^2\gamma_{\xi}^2\Big({\rm Id}-\frac{\mathring{R}_0}{\rho_0}\Big)\eta^2_{(\xi)}\Big(-e^{2i \lambda_1 \xi^{\bot}\cdot x}-e^{-2i\lambda_1 \xi^{\bot}\cdot x}\Big)dx}_{(4)}\\
&&+\underbrace{\frac{1}{2}\int_{{\rm  T}^2} \sum_{j \geq 1} \sum_{\xi \in \Lambda_{(j)}^+} a_{(\xi,j)}^2 \eta_{(\xi)}^2\Big|W_{(\xi)}+W_{(-\xi)}\Big|^2dx}_{(5)}.
\eeno
Thus, combing the definition (\ref{d:definition of zero amplitude}) of $\rho_0$, we obtain
\beno
\int_{{\rm T}^2}|v_1(t,x)|^2dx&=&\int_{{\rm T}^2}|v_0(t,x)|^2dx+2\rho_0\int_{{\rm T}^2}\chi_0^2dx+E(t)_{err}\\
&=&e(t)\Big(1-\frac{\delta}{2}\Big)+E(t)_{err}.
\eeno
Next, we will show that by choosing the parameter $\lambda_1$ sufficiently large, there holds
\beno
\big|E(t)_{err}\big|\leq \frac{\delta e(t)}{8}, \quad \forall t\in [0, 1],
\eeno
thus, we obtain
\beno
\Big|e(t)\Big(1-\frac{\delta}{2}\Big)-\int_{{\rm T}^2}|v_1(t,x)|^2dx\Big|\leq \frac{\delta e(t)}{8}, \quad \forall t\in [0, 1],
\eeno
which give (\ref{e:energy estimate new}).

{\bf Estimate on $E(t)_{err}$:} We estimate $E(t)_{err}$ term by term.

{\bf Estimate on (1):} By Lemma \ref{e:improved $L^p$ inequality} and the $L^p$ estimate in Proposition \ref{e:estimate on the pertirbation}, we deduce that for every $t\in [0,1]$ and $p> 1$
\beno
|(1)|\leq  C(v_0)\|w_{1}(t)\|_{L^1} \leq C(\mathring{R}_0, v_0, \delta, p)r_1^{1-\frac{2}{p}}.
\eeno


{\bf Estimate on (2):}
By Proposition \ref{e:estimate on the pertirbation}, it's direct to get
\beno
\int_{{\rm T}^2}\big|w^c_{1}+w^t_{1}\big|^2 dx&\leq& C(\mathring{R}_0, \delta) \big((r_1 \sigma_1)^2+(r_1 \mu_1^{-1})^2\big),\\
\Big|\int_{{\rm T}^2}w_{1}^p\cdot \big(w^c_{1}+w^t_{1}\big)dx\Big| &\leq& C(\mathring{R}_0, \delta) \big(r_1 \sigma_1+r_1 \mu_1^{-1}\big).
\eeno
Then, by Lemma \ref{e:mean value estimate} and Proposition \ref{e:estimate on the pertirbation}, we deduce that for any $t\in [0,1]$ and $p> 1$
\begin{align*}
|II| & \leq C(\mathring{R}_0, \delta, p)\lambda_1^{-1}\|w_1\|_{L^{2p}}^2 \leq C(\mathring{R}_0, \delta, p)\lambda_1^{-1}r_1^{2(1-\frac{1}{p})}.
\end{align*}
Summing these term, we obtain
\beno
|(2)|\leq C(\mathring{R}_0, \delta, p)\Big(r_1 \sigma_1 + r_1 \mu_1^{-1} + \lambda_1^{-1}r_1^{2(1-\frac{1}{p})}\Big).
\eeno

{\bf Estimate on (3):} Notice that
\beno
\fint_{{\rm T}^2}\big(\eta^2_{(\xi)}-1\big)=0,
\eeno
thus,
\beno
{\rm P}_{\geq \frac{\lambda_1 \sigma_1}{2}}\big(\eta^2_{(\xi)}-1\big)= \big(\eta^2_{(\xi)}-1\big).
\eeno
Lemma \ref{e:mean value estimate}, Lemma \ref{e:pointwise estimate on cutoff function} and Lemma \ref{e:$L^2$ and derivative estimate on cutoff function} implies
\beno
|(3)|\leq C(\mathring{R}_0, \delta) \frac{1}{\lambda_1 \sigma_1}.
\eeno

{\bf Estimate on (4):} Notice the fact
\beno
{\rm P}_{\geq \frac{\lambda_1}{2}}\Big(\eta^2_{(\xi)}\Big(-e^{2i \lambda_1 \xi^{\bot}\cdot x}-e^{-2i\lambda_1 \xi^{\bot}\cdot x}\Big)\Big)=\eta^2_{(\xi)}\Big(-e^{2i \lambda_1 \xi^{\bot}\cdot x}-e^{-2i\lambda_1 \xi^{\bot}\cdot x}\Big).
\eeno
Thus, Lemma \ref{e:mean value estimate}, Lemma \ref{e:pointwise estimate on cutoff function} and Lemma \ref{e:$L^2$ and derivative estimate on cutoff function} gives
\beno
|(4)|\leq \frac{C(\mathring{R}_0, \delta)}{\lambda_1 }.
\eeno

{\bf Estimate on (5):} Recalling (\ref{e:energy of first odd part}), we obtain
\beno
|(5)|\leq \frac{\delta}{20}+  C(\mathring{R}_0, \delta) \frac{1}{\lambda_1 \sigma_1}\leq \frac{\delta e(t)}{20}+  C(\mathring{R}_0, \delta) \frac{1}{\lambda_1 \sigma_1}.
\eeno

Finally, collecting estimate (1)-(5), noticing the parameter relationship (\ref{r:relationship of papameter}), taking $p$ sufficiently close to 1 and parameter $\lambda_1$ sufficiently large, we arrive at
\beno
E_{err}(t)\leq \frac{\delta e(t)}{8}.
\eeno
This completes the proof.

{\bf Acknowledgments.}
The first author is supported in part by NSFC Grants 11601258.6.
The second author is supported by the fundamental research funds of Shandong university under Grant 11140078614006. The third author is partially supported by the Chinese NSF under Grant 11471320 and 11631008.

\end{document}